\documentclass[final,3p,12pt]{elsarticle}

\usepackage{amssymb}
\newcommand{\dawesspace}{\ \hspace{-5.0cm}}  % Jon's terrible layout command for align* environments.

\usepackage[utf8]{inputenc} % allow utf-8 input
\usepackage[T1]{fontenc}    % use 8-bit T1 fonts
\usepackage{hyperref}       % hyperlinks
\usepackage{url}            % simple URL typesetting
\usepackage{booktabs}       % professional-quality tables
\usepackage{amsfonts}       % blackboard math symbols
\usepackage{nicefrac}       % compact symbols for 1/2, etc.
\usepackage{microtype}      % microtypography
\usepackage{lipsum}
\usepackage{amsmath}
\usepackage{amsthm}
\usepackage{amssymb}
\usepackage{mathtools}
\newtheorem{theorem}{Theorem}
\newtheorem{lemma}[theorem]{Lemma}

\theoremstyle{definition}
\newtheorem{definition}[theorem]{Definition}
\theoremstyle{definition}
\newtheorem{remark}[theorem]{Remark}
\numberwithin{theorem}{section}  % [jhpd] was subsection - but this looks very picky I feel...
\usepackage{graphicx}
\usepackage{subcaption}
\usepackage{natbib}
\DeclareMathOperator*{\argmin}{arg\,min}
\usepackage{xcolor}
\usepackage{listings}
\usepackage{xparse}
\NewDocumentCommand{\codeword}{v}{%
\texttt{\textcolor{black}{#1}}%
}

\newcommand{\ba}{\begin{eqnarray}}
\newcommand{\ea}{\end{eqnarray}}
\newcommand{\nn}{\nonumber}

\usepackage[noend]{algpseudocode}
\makeatletter
\def\BState{\State\hskip-\ALG@thistlm}
\makeatother
\algdef{SE}[SUBALG]{Indent}{EndIndent}{}{\algorithmicend\ }%
\algtext*{Indent}
\algtext*{EndIndent}

\usepackage[ruled]{algorithm2e}
\makeatletter
\newenvironment{algoproc}[1][]
  {\renewcommand{\algorithmcfname}{Procedure}%
   \begin{algorithm}[#1]
   \long\def\@caption##1[##2]##3{%
     \par
     \begingroup\@parboxrestore
     \if@minipage\@setminipage\fi
     \normalsize \@makecaption{\AlCapSty{\AlCapFnt\algorithmcfname}}{\ignorespaces ##3}%
     \par\endgroup
   }}
  {\end{algorithm}}
\makeatother

\journal{Neural Networks}

\begin{document}
\lstset{language=Python,keywordstyle={\bfseries \color{black}}}

%% \linenumbers

\begin{frontmatter}

%% Title, authors and addresses

%% use the tnoteref command within \title for footnotes;
%% use the tnotetext command for theassociated footnote;
%% use the fnref command within \author or \address for footnotes;
%% use the fntext command for theassociated footnote;
%% use the corref command within \author for corresponding author footnotes;
%% use the cortext command for theassociated footnote;
%% use the ead command for the email address,
%% and the form \ead[url] for the home page:
%% \title{Title\tnoteref{label1}}
%% \tnotetext[label1]{}
%% \author{Name\corref{cor1}\fnref{label2}}
%% \ead{email address}
%% \ead[url]{home page}
%% \fntext[label2]{}
%% \cortext[cor1]{}
%% \affiliation{organization={},
%%             addressline={},
%%             city={},
%%             postcode={},
%%             state={},
%%             country={}}
%% \fntext[label3]{}

\title{Using Echo State Networks to Approximate Value Functions for Control}

%% use optional labels to link authors explicitly to addresses:
%% \author[label1,label2]{}
%% \affiliation[label1]{organization={},
%%             addressline={},
%%             city={},
%%             postcode={},
%%             state={},
%%             country={}}
%%
%% \affiliation[label2]{organization={},
%%             addressline={},
%%             city={},
%%             postcode={},
%%             state={},
%%             country={}}

\author[inst1]{Allen G. Hart}
\author[inst1]{Kevin R. Olding}
\author[inst1]{Alexander M.G. Cox}
\author[inst2]{Olga Isupova}
\author[inst1]{Jonathan H.P. Dawes}

\affiliation[inst1]{organization={Department of Mathematical Sciences,
University of Bath}, %Department and Organization
            %addressline={}, 
            city={Bath},
            postcode={BA2 7AY}, 
            %state={},
            country={UK}}

\affiliation[inst2]{organization={Department of Computer Science,
University of Bath},%Department and Organization
            %addressline={}, 
            city={Bath},
            postcode={BA2 7AY}, 
            %state={},
            country={UK}}

%% \documentclass[twoside,11pt]{article}

% % Any additional packages needed should be included after jmlr2e.
% % Note that jmlr2e.sty includes epsfig, amssymb, natbib and graphicx,
% % and defines many common macros, such as 'proof' and 'example'.
% %
% % It also sets the bibliographystyle to plainnat; for more information on
% % natbib citation styles, see the natbib documentation, a copy of which
% % is archived at http://www.jmlr.org/format/natbib.pdf
% \usepackage{jmlr2e}
% % Definitions of handy macros can go here

% \newcommand{\dataset}{{\cal D}}
% \newcommand{\fracpartial}[2]{\frac{\partial #1}{\partial  #2}}
% \newcommand{\dawesspace}{\ \hspace{-5.0cm}}  % Jon's terrible layout command for align* environments.

% Heading arguments are {volume}{year}{pages}{submitted}{published}{author-full-names}
% \jmlrheading{1}{2000}{1-48}{4/00}{10/00}{Hart, Olding, Cox, Isupova, and Dawes}
% Short headings should be running head and authors last names

%\ShortHeadings{Echo State Networks for Reinforcement Learning}{Hart, Olding, Cox, Isupova, and Dawes}
%\firstpageno{1}

\begin{abstract}
An Echo State Network (ESN) is a type of single-layer recurrent neural network with randomly-chosen internal weights and a trainable output layer. We prove under mild conditions that a sufficiently large Echo State Network can approximate the value function of a broad class of stochastic and deterministic control problems. Such control problems are generally non-Markovian.

%We discuss how a reinforcement learning agent supported by an ESN will eventually learn the value function associated to the control strategy the agent is deploying.
We describe how the ESN can form the basis for novel and computationally efficient reinforcement learning algorithms in a non-Markovian framework. We demonstrate this theory with two examples. In the first, we use an ESN to solve a deterministic, partially observed, control problem which is a simple game we call `Bee World'. In the second example, we consider a stochastic control problem inspired by a market making problem in mathematical finance. In both cases we can compare the dynamics of the algorithms with analytic solutions to show that even after only a single reinforcement policy iteration the algorithms arrive at a good policy.
\end{abstract}

%%Graphical abstract
%\begin{graphicalabstract}
%\includegraphics{grabs}
%\end{graphicalabstract}

%%Research highlights
% \begin{highlights}
% \item Echo State Networks (ESNs) are presented in the context of filters and functionals.
% \item ESNs trained by least squares regression on a stationary ergodic process can learn an arbitrary target functional.
% \item Applying this to reinforcement learning, we show that ESNs adopting a stationary ergodic policy trained by least squares regression can learn the value functional. We prove related results for both offline and online learning.
% \item We demonstrate some of this theory on two simple reinforcement learning problems which admit analytic solutions.
% \end{highlights}

\begin{keyword}
    Liquid State Machines \sep
    Reservoir Computing \sep
    Stochastic Optimal Control \sep
    Mathematical Finance \sep
    Reinforcement Learning
%% PACS codes here, in the form: \PACS code \sep code
\PACS 0000 \sep 1111
%% MSC codes here, in the form: \MSC code \sep code
%% or \MSC[2008] code \sep code (2000 is the default)
\MSC 0000 \sep 1111
\end{keyword}

%\begin{keyword}
%% keywords here, in the form: keyword \sep keyword
%keyword one \sep keyword two
%% PACS codes here, in the form: \PACS code \sep code
%\PACS 0000 \sep 1111
%% MSC codes here, in the form: \MSC code \sep code
%% or \MSC[2008] code \sep code (2000 is the default)
%\MSC 0000 \sep 1111
%\end{keyword}

\end{frontmatter}

\newpage

\section{Introduction}

An Echo State Network (ESN) is a special type of single-layer recurrent neural network introduced at the turn of the millennium by \cite{Jaeger2001} and \cite{doi:10.1162/089976602760407955} to study time series. Training is fast because the training step involves only the selection of weights in the output layer rather than updating the internal weights in the recurrent layer. Furthermore, the simple formulation of ESNs renders them amenable to mathematical analysis.
Given a time series $z_k$ (where $k$ is the discrete time index) of $d$-dimensional data points, an ESN is set up as follows. We randomly generate a $n \times n$ reservoir matrix $\boldsymbol{A}$, a $n \times d$ input matrix $\boldsymbol{C}$ and a $n \times 1$ bias vector $\boldsymbol{\zeta}$. Then we iteratively generate a sequence of $n$-dimensional \emph{reservoir state} vectors $x_k$ according to
\begin{align*}
    x_{k+1} = \sigma(\boldsymbol{A} x_k + \boldsymbol{C} z_k + \boldsymbol{\zeta})
\end{align*}
where $\sigma(x)_i = \max(0,x_i)$ is the rectified linear unit (ReLU) activation function applied component-wise to the $n$-dimensional vector $x$. Observe that the $k$th reservoir state $x_k$ depends on all past data-points $\ldots, z_{k-2}, z_{k-1}$ and therefore captures non-Markovian temporal correlations in the data. If the 2-norm of the reservoir matrix satisfies $\lVert \boldsymbol{A} \lVert_2 < 1$ then as $n$ tends to infinity, the influence on the reservoir state $x_{k+n}$ of the data points $\ldots, z_{k-2},z_{k-1}$ in the distant past  becomes arbitrarily small. This is called the \emph{fading memory property} and is closely related to the \emph{echo state property} (ESP) introduced in the context of ESNs by \cite{Jaeger2001}.
The ESP is the statement that the sequence of reservoir states $(x_k)_{k \in \mathbb{Z}}$ is, for a given input data sequence $(z_k)_{k \in \mathbb{Z}}$, uniquely determined. We can interpret the reservoir state vectors as the latent vectors which encode the infinite past observations in lower dimensional form.

When an ESN has the ESP, it can be applied to a class of supervised learning problems where we have a time series of $d$ dimensional data points $r_k$, called \emph{targets}, that depend on all previous input time series data $\ldots, z_{k-3}, z_{k-2}, z_{k-1}$ and we seek to learn the relationship between the sequence of past states and the target for each $k$. We can train an ESN to solve this problem by finding the $m \times d$ matrix $W$ that minimises
\begin{align*}
    \sum_{k=0}^{\ell-1} \lVert W^{\top}x_{k} - r_{k} \rVert^2 + \lambda\lVert W \rVert^2,
\end{align*}
where $\ell$ is the number of labelled data points, and $\lambda > 0$ is the Tikhonov regularisation (a.k.a. ridge regression) parameter. Throughout this paper, $\lVert \cdot \rVert$ denotes the matrix 2-norm, vector 2-norm or absolute value, depending on whether the input is a matrix, vector, or scalar, respectively.

This minimisation problem can be solved using regularised linear least squares regression, and hence we can both obtain $W$ quickly, and guarantee that $W$ is the global optimum. This compares extremely favourably with training a (deep) neural network with stochastic gradient descent and backpropagation which takes considerably longer, and may not converge to the global optimum \citep{General_Value_Function_Networks}.

Despite the training procedure being entirely linear, ESNs are universal approximators, and can therefore model arbitrarily complex relationships between the sequence of past data points and the targets. This is made formal in a recent result by \cite{Gonon2020} that we review here and then build on. We emphasise that not only are ESNs theoretically very promising, they have performed remarkably well in practice on problems ranging from seizure detection, to robot control,
handwriting recognition, and financial forecasting, where ESNs have won competitions \citep{LUKOSEVICIUS2009127},  \citep{LukoseviciusMantas2012RCT}, \citep{RodanA2011MCES}, \citep{NIPS2010_4056}. Impressively, ESNs outperformed RNNs and LSTMs at a chaotic time series prediction task by a factor of over 2400 \citep{Jaeger78}. ESNs have also proved themselves competitive in various tasks in reinforcement learning \citep{10.1007/11840817_86} and control \citep{Pietz_2021}.

Even in cases where practitioners prefer to use other recurrent neural networks (RNNs), such as Long Short Term Memory networks (LSTMs), the rigorous theory of ESNs should prove useful in architecture design. In \cite{Random_Weights}, it is shown that different deep neural network architectures can be ranked by randomly initialising the internal weights and training only the outer weights by linear regression. Once the best performing architecture (with random internal weights) has been identified, the authors then train the internal weights of the highest ranking architecture. This is much faster than training the internal weights (a nonlinear problem) for every architecture. The ranking of architectures with random internal weights closely approximates the ranking of architectures with optimised internal weights. From our point of view, the authors are essentially approximating fully trained networks with (non-recurrent) ESNs.

In a sequence of papers, \cite{GRIGORYEVA2018495}, \cite{JMLR:v20:19-150}, and \cite{Gonon2020} recently analysed ESNs in the context of nonlinear filters and functionals.
Roughly speaking, a \emph{filter} $U$ is a map from a bi-infinite sequence $\ldots , z_{-2},z_{-1},z_0, z_1, z_2, \ldots $ of real vectors to another bi-infinite sequence of real vectors $\ldots, x_{-2}, x_{-1},x_0, x_1, x_2, \ldots $, and a \emph{functional} $H$ maps a bi-infinite sequence $\ldots, z_{-2}, z_{-1}, z_0, z_1, z_2, \ldots $ of real vectors to a single real vector or number. We can view an ESN as a filter that maps an input sequence $\ldots, z_{-2}, z_{1}, z_0, z_1, z_2,\ldots $ to a reservoir sequence $\ldots, x_{-2},x_{-1},x_0, x_1, x_2, \ldots $, or a funtional that maps $\ldots, z_{-2}, z_{1}, z_0, z_1, z_2,\ldots $ to the lone reservoir state $x_0$. The theory of filters and functionals is therefore a natural theoretical setting for ESNs.  Within this theory, this paper presents three novel results. 

Our first result assumes that we have a time series of data $z_k$ and a set of targets $r_k$ that depend on all previous data points $\ldots,z_{k-2}, z_{k-1}$ via a functional $\mathcal{R}$ which sends infinite sequences of data points to targets. We then have a supervised learning problem of finding the relationship between the data and targets. In the special case that $z_k = r_k$, this problem is \emph{time series forecasting}. Our first novel result states that if we have sufficiently many data points $z_k$, drawn from a stationary, ergodic, and bounded process $\boldsymbol{Z}$, which need not be Markovian, and we obtain $W$ using regularised linear least squares, then a sufficiently large ESN will approximate, as closely as required, the functional $\mathcal{R}$ sending inputs $\ldots,z_{k-2}, z_{k-1}$ to the targets $r_k$. 

This result has applications in the \emph{statistical inference of dynamical systems}, which was recently reviewed by \cite{mcgoff2015}. This area of research is especially focused on statistical inference (i.e learning) of stationary ergodic processes. Furthermore, we can use this result in the context of reinforcement learning (RL) and optimal control. We envisage an agent operating under a given policy in the parlance of reinforcement learning or control in the parlance of control theory that generates a sequence of (reward, action, observation) triples $z_k = (r_k,a_k,\omega_k)$. Then the functional $V$ that maps previous (reward, action, observation) triples $\ldots , z_{k-2}, z_{k-1}$ to rewards $z_k$ models the reward functional arbitrarily well. The set up does not assume the RL problem is Markovian, and allows for a continuous state space.

Our second novel result generalises the first, and encompasses the case where the functional $V$ is the value functional of a stochastic control process, or Partially Observed Markov Decision Process (POMDP). By training an ESN to approximate the value functional, we establish a stepping stone toward developing an offline reinforcement learning algorithm supported by an ESN that can solve a large class of control problems. Moreover, since ESNs are recurrent, they can be used for non-Markovian problems, where a reinforcement learning agent must exploit its \emph{memory} of past observations, actions and rewards. Our third result is presented in the context of building an online reinforcement algorithm that can, under certain conditions, determine the optimal value function for a given policy. 

These results are part of a general push to take machine learning ideas typically applied to (partially observed) Markov processes and generalising them to hold on stationary ergodic processes. We can see for example \cite{9174045} consider to clustering problems typically defined Markov processes applied to stationary ergodic processes.

We demonstrate some of these theoretical results numerically on two examples. The first is a deterministic game which we call `Bee World'. The goal of the game for the bee is to navigate a time varying distribution of nectar in order to maximise the total future discounted value of the nectar acquired over all future time. The optimal trajectory can be found explicitly via the calculus of variations but the constraint that the bee has a maximum speed of flight leads to unexpectedly complicated solution paths; it therefore provides a straightforward but not entirely trivial control problem. Since the bee does not have access to the entire state space, and only observes the nectar it collects at each moment in time, the problem is therefore a partially observed Markov Decision Process which requires memory of the past to solve. We demonstrate how a simple and easily-configurable reinforcement learning algorithm supported by an ESN can learn to play Bee World with respectable skill.

The second numerical example is inspired by a market making problem in mathematical finance. The mathematical formulation of this problem reduces to a seeking to control a one dimensional Brownian motion so that it stays near the origin. The cost of straying from the origin is quadratic in the distance from the origin, and the cost of applying a push toward the origin is quadratic in the strength of the push. The market maker must therefore balance the cost of applying the control against the cost of allowing the motion to drift too far from the origin. We briefly discuss the financial motivation for this problem, then solve it analytically in continuous and discrete time. The set up most commonly seen in the literature is continuous time, but only in discrete time is the problem suitable for an ESN. We then compare the optimal discrete time solution to a solution learned by a reinforcement learning agent supported by an ESN.

Finally, we note that our approach to the Market making problem is loosely related to the recent paper by \cite{Pietz_2021} who introduce QuaSiModO: Quantization-Simulation-Modeling-Optimization. These authors analyse the interplay between the following four aspects:
\begin{enumerate}
    \item Quantising the action space $\mathcal{A}$.
    \item Simulating a system under a given control/policy.
    \item Modelling the full system given a partial/full observation of the state space.
    \item Optimising the control/policy.
\end{enumerate}

The structure of the remainder of the paper closely follows the summary of results presented above. In section~\ref{theory_of_filters} we set up the mathematical formalism for ESNs that we wish then to approximate. Section~\ref{training_ESNs_with_least_squares} introduces our novel theoretical results, while sections~\ref{Bee_World} and~\ref{a_stochastic_optimal_control_problem} respectively present applications to the deterministic (`Bee World'), and then the stochastic (`market maker') optimal control problems. We conclude in section~\ref{Conclusions}.

\section{Background}%Theoretical Results}
\label{theory_of_filters}
In this section, we introduce the theory and notation of nonlinear filters (in relation to ESNs) developed by \cite{GRIGORYEVA2018495}, \cite{JMLR:v20:19-150}, and \cite{Gonon2020}. First, we denote by  $(\mathbb{R}^d)^{\mathbb{Z}}$ the set of maps with domain $\mathbb{Z}$ and codomain $\mathbb{R}^d$. This is the set of bi-infinite $\mathbb{R}^d$--valued real sequences.

A \emph{filter} is a map $U : (\mathbb{R}^d)^{\mathbb{Z}} \to (\mathbb{R}^n)^{\mathbb{Z}}$. A filter $U$ is called \emph{causal} if inputs from the past and present $\ldots, z_{-2}, z_{-1}, z_0$ contribute to $U(z)$ but states in the future $z_1, z_2 \ldots$ do not. More formally $U$ is casual if $\forall \ z,y \in (\mathbb{R}^d)^{\mathbb{Z}}$ that satisfy $z_k = y_k \ \forall \ k \leq 0$ it follows that $U(z) = U(y)$. We define the time shift filter $T : (\mathbb{R}^d)^{\mathbb{Z}} \to (\mathbb{R}^n)^{\mathbb{Z}}$ by $T(z)_k = T(z)_{k+1}$ which we interpret as the map that steps forward one unit of time. A filter $U$ is called \emph{time invariant} if $U$ commutes with the time shift operator $T$. If $U$ is causal and time invariant filter then we call $U$ a causal time invariant (CTI) filter.

A functional is a map $H : (\mathbb{R}^d)^{\mathbb{Z}} \to \mathbb{R}^n$. In \cite{JMLR:v20:19-150} it is shown that there is a bijection between the space of CTI filters and the space of functionals. To see this, take a functional $H$ and define the $k$th term of the associated filter $U$ via $U(z)_k = HT^k(z)$. Conversely, given a filter $U$, the associated functional $H$ is given by $H(z) = U(z)_0$

We can view an ESN as a CTI filter from the space of input sequences $\ldots , z_{-1}, z_0, z_1, \ldots$ to the space of reservoir sequences $\ldots , x_{-1}, x_0, x_1, \ldots$.  To make this connection between ESNs and filters formal, we will first present a generalisation of an Echo State Network called a reservoir system. 
\begin{definition}
    (Reservoir system) Let $F : \mathbb{R}^n \times \mathbb{R}^d \to \mathbb{R}^n$ and $h : \mathbb{R}^n \to \mathbb{R}^s$. Then we call the following system of equations
    \begin{align}
        x_{k+1} &= F(x_k,z_k) \label{reservoir_system_eqns} \\
        r_k &= h(x_k) \nonumber
    \end{align}
    a reservoir system.
\end{definition}
\begin{remark}
    We can see that if
\begin{align*}
    F(x,z) &= \sigma(\boldsymbol{A} x + \boldsymbol{C} z + \boldsymbol{\zeta}) \\
    h(x) &= W^{\top} x 
\end{align*}
then we retrieve an ESN with $n \times n$ reservoir matrix $\boldsymbol{A}$, $n \times d$ input matrix $\boldsymbol{C}$, bias vector $\boldsymbol{\zeta} \in \mathbb{R}^n$, linear output layer $W \in \mathbb{R}^n$, and activation function $\sigma  = \text{ReLU}$, defined in the introduction.
\end{remark}

 We require that the reservoir system induces a \emph{unique} filter from the input sequence to the reservoir sequence. This property is the Echo State Property that we briefly mentioned in the introduction.

\begin{definition}
    (Echo State Property \citep{Jaeger2001}) A reservoir system has the Echo State Property (ESP) if for any $(z_k)_{k \in \mathbb{Z}} \in (\mathbb{R}^d)^{\mathbb{Z}}$ there exists a unique $(x_k)_{k \in \mathbb{Z}} \in (\mathbb{R}^n)^{\mathbb{Z}}$ that satisfy the equations of the reservoir system~\eqref{reservoir_system_eqns}.
\end{definition}

To any reservoir system with the Echo State property we can associate a unique CTI reservoir filter $U : (\mathbb{R}^d)^{\mathbb{Z}} \to (\mathbb{R}^n)^{\mathbb{Z}}$ defined by $U(z) = x$. To this reservoir filter, we may assign a CTI reservoir functional $H : (\mathbb{R}^m)^{\mathbb{Z}} \to \mathbb{R}^d$ defined by $H(z) = x_0$. In a supervised learning context, we have a time series of data points $\ldots ,z_{-2}, z_{-1}, z_0$ and a time series of targets $\ldots ,r_{-1}, r_0$ that each depend on \emph{all} previous data points.
The output functional $ h \circ H : (\mathbb{R}^d)^{\mathbb{Z}} \to \mathbb{R}$ is the map we use to approximate the relationship between the data and the targets, so $h \circ H(\ldots, z_{-2},z_{-1},z_0,z_1,z_2,\ldots ) \approx r_k$. Note that $h \circ H$ is \emph{causal}, so does not peer into the future and use data $z_1,z_2,\ldots $ that have not yet been revealed. 
When the reservoir system is an ESN, the map $h$ is the linear map $W^{\top}$ obtained by least squares ridge regression, so that $W^{\top}H(\ldots, z_{-2},z_{-1}, z_0, z_{1}, \ldots) \approx r_k$. 
We assume there exists a \emph{true} map from the data to the targets that we label $\mathcal{R} :\mathbb{R}^{\mathbb{Z}} \to \mathbb{R}$ so that $\mathcal{R}(\ldots, z_{-2},z_{-1}, z_0, z_1, \ldots) = r_k$. Our goal is to find $W$ such that $W^{\top}H \approx \mathcal{R}$.

\begin{definition}
    (ESN filter and functional) If an ESN has the ESP then we will write $H^{\boldsymbol{A},\boldsymbol{C},\boldsymbol{\zeta}}$ to denote the reservoir functional associated to an ESN with parameters $\boldsymbol{A},\boldsymbol{C}$ and $\boldsymbol{\zeta}$. We will also write $H^{\boldsymbol{A},\boldsymbol{C},\boldsymbol{\zeta}}_W$ to denote the output functional $W^{\top} H^{\boldsymbol{A},\boldsymbol{C},\boldsymbol{\zeta}}$ (defined by left multiplication of $H^{\boldsymbol{A},\boldsymbol{C},\boldsymbol{\zeta}}$ by the linear readout layer) 
    %so that our notation is consistent with \cite{Gonon2020}.
\end{definition}

Next, we will present a procedure, introduced by \cite{Gonon2020}, for randomly generating the ESN's internal weights $\boldsymbol{A},\boldsymbol{C}$ and biases $\boldsymbol{\zeta}$, which ensures the ESN has ESP and allows for the universal approximation of target functionals $\mathcal{R}$.
The procedure differs from the procedure commonly seen in the literature, where $\boldsymbol{A}, \boldsymbol{C},\boldsymbol{\zeta}$ are populated with i.i.d Gaussians, or i.i.d uniform deviates, and then $\boldsymbol{A}$ is rescaled so that its 2-norm (or spectral radius) is less than 1. Furthermore, the procedure introduced by \cite{Gonon2020} depends on some details of the input process, which must satisfy mild conditions stated below.

\begin{definition}
    (Admissible input process) A $(\mathbb{R}^d)^\mathbb{Z}$ valued random variable $\boldsymbol{Z}$ is called an admissible process if for any $T \in \mathbb{N}$ there exists $M_T>0$ such that for all $k \in \mathbb{Z}$
    \begin{align}
        \label{admissible}
        \lVert \boldsymbol{Z}_{k-T} , \boldsymbol{Z}_{k-T+1} , \ldots  , \boldsymbol{Z}_k \rVert \leq M_T
    \end{align}
    Lebesgue-almost surely.
\end{definition}

We will now present a procedure by which the matrices $\boldsymbol{A},\boldsymbol{C},\boldsymbol{\zeta}$ are randomly generated. 

\begin{algoproc}[H]
  \SetAlgoLined
  \caption{Initialising the random weights of an ESN.}
    Let $N \in \mathbb{N}$, $R > 0$ be the input parameters for the procedure. Suppose that $\boldsymbol{Z}$ is an admissible input process. Consequently, for any $T_0 \in \mathbb{N}$ there exists $M_{T_0}$ such that (for $k=0$ in (\ref{admissible}))
    \begin{align*}
        \lVert \boldsymbol{Z}_{-{T_0}} , \boldsymbol{Z}_{-T_0+1} , \ldots  , \boldsymbol{Z}_0 \rVert \leq M_T
    \end{align*}
    Lebesgue-almost surely.
    Then, for a given $T_0$, we initialise the ESN reservoir matrix $\boldsymbol{A}$, input matrix $\boldsymbol{C}$, and biases $\boldsymbol{\zeta}$ according to the following procedure.
    \begin{enumerate}
        \item Draw $N$ i.i.d. samples $\boldsymbol{A}_1 , \ldots  , \boldsymbol{A}_N$ from the uniform distribution on $B_R \subset \mathbb{R}^{d(T_)+1)}$ where $B_R$ is the ball of radius $R$ and centre 0, and draw $N$ i.i.d. samples $\boldsymbol{\zeta}_1 , \ldots  \boldsymbol{\zeta}_N$ from the uniform distribution on $[-\max(M_{T_0} R, 1) , \max(M_{T_0} R, 1)]$.
        \item Let $S$ and $c$ be shift matrices defined
        \begin{align*}
            S = 
            \begin{bmatrix}
                0_{d,dT_0} & 0_{d,d} \\
                I_{dT_0} & 0_{dT_0,d}
            \end{bmatrix}
            \qquad 
            c = 
            \begin{bmatrix}
                I_{d} \\
                0_{dT_0,d}
            \end{bmatrix}
        \end{align*}
        and set
        \begin{align*}
            \boldsymbol{a} =
            \begin{bmatrix}
                \boldsymbol{A}_1^{\top} \\
                \boldsymbol{A}_2^{\top} \\
                \vdots \\
                \boldsymbol{A}_N^{\top}
            \end{bmatrix}
            \qquad 
            \boldsymbol{\bar{A}} =
            \begin{bmatrix}
                S & 0_{d(T_0+1),N} \\
                \boldsymbol{a}S & 0_{N,N} 
            \end{bmatrix}
                \\
            \boldsymbol{\bar{C}} =
            \begin{bmatrix}
                c \\
                \boldsymbol{a}c 
            \end{bmatrix}
            \qquad
            \boldsymbol{\bar{\zeta}} = 
            \begin{bmatrix}
                0_{d(T_0+1)} \\
                \boldsymbol{\zeta}_1 \\
                \vdots \\
                \boldsymbol{\zeta}_N
            \end{bmatrix}
        \end{align*}
        so that
        \begin{align*}
            \boldsymbol{A} = 
            \begin{bmatrix}
                \boldsymbol{\bar{A}} & -\boldsymbol{\bar{A}} \\
                -\boldsymbol{\bar{A}} & \boldsymbol{\bar{A}}
            \end{bmatrix}
            \qquad
            \boldsymbol{C} = 
            \begin{bmatrix}
                \boldsymbol{\bar{C}} \\
                -\boldsymbol{\bar{C}}
            \end{bmatrix}
            \qquad
            \boldsymbol{\zeta} = 
            \begin{bmatrix}
                \boldsymbol{\bar{\zeta}} \\
                -\boldsymbol{\bar{\zeta}}
            \end{bmatrix}.
        \end{align*}
    \end{enumerate}
    \label{proc:procedure}
\end{algoproc}

We are now ready to present the key result by \cite{Gonon2020}, (which generalises a result by \cite{HART2020234}) and which holds in the following supervised learning context. Given time series data $z_k$ (from an admissible process $\boldsymbol{Z}$) and a time series of targets $r_k$ depending on all previous data $\ldots,z_{k-2}, z_{k-1}$ we wish to approximate the functional that sends $\ldots,z_{k-2}, z_{k-1}$ to $r_k$. We will denote this functional $\mathcal{R}$. The problem of approximating $\mathcal{R}$ given the data and targets is a supervised learning problem. The result can be summarised as follows. Suppose we have an ESN with weights $\boldsymbol{A},\boldsymbol{C}$ and biases $\boldsymbol{\zeta}$ randomly generated by procedure \ref{proc:procedure}. Then, the ESN admits a linear readout matrix $W$ for which the ESN equipped with the matrix $W$ (denoted $H^{\boldsymbol{A},\boldsymbol{C},\boldsymbol{\zeta}}_W$) approximates the relationship $\mathcal{R}$ between data points $\ldots,z_{k-2}, z_{k-1}$ and targets $r_k$ as closely as is required.

\begin{theorem}
[\cite{Gonon2020}] Suppose that $\boldsymbol{Z}$ is an admissible input process. Let $\mathcal{R} : (D_n)^{\mathbb{Z}} \to \mathbb{R}$ (where $D_n$ is a compact subset of $\mathbb{R}^n$) be CTI and measurable with respect to some measure $\mu$ such that $\mathbb{E}_{\mu}[|\mathcal{R}(\boldsymbol{Z})|^2] < \infty$.

Then for any $\epsilon > 0$ and $\delta \in (0,1)$ there exists $N,T_0 \in \mathbb{N}$, $R > 0$ such that, with probability $(1-\delta)$, the ESN with parameters $\boldsymbol{A},\boldsymbol{C},\boldsymbol{\zeta}$ generated by the procedure in definition~\ref{proc:procedure} (with inputs $N,T_0,R$) has the ESP and admits a readout layer $W \in \mathbb{R}^{2(d(T_0+1)+N)}$ such that
\begin{align*}
\bigg(\mathbb{E}_{\mu}\left[ \left. \left\lVert H^{\boldsymbol{A},\boldsymbol{C},\boldsymbol{\zeta}}_W(\boldsymbol{Z}) - \mathcal{R}(\boldsymbol{Z}) \right\rVert^2 \, \right| \, \boldsymbol{A} , \boldsymbol{C} , \boldsymbol{\zeta} \right]\bigg)^{1/2}
\hspace{-0.25cm} :=
\bigg( \int_{(\mathbb{R}^d)^{\mathbb{Z}}} \left\lVert H^{\boldsymbol{A},\boldsymbol{C},\boldsymbol{\zeta}}_W(z) - \mathcal{R}(z)\right\rVert^2 d\mu(z) \bigg)^{1/2}
\hspace{-0.25cm} < \epsilon .
\end{align*}
    \label{general_ESN_approximation_theorem}
\end{theorem}

\section{Novel results for ESNs}
\label{training_ESNs_with_least_squares}

Theorem \ref{general_ESN_approximation_theorem} is an existence result stating that \emph{there exists} a linear readout layer $W$ yielding an arbitrarily good approximation. Our first novel contribution is to strengthen the result under additional assumptions. The new result states that, given a sufficiently large ESN and sufficiently many training data $z_k$ drawn from a stationary, ergodic and bounded process $\boldsymbol{Z}$, if we train an ESN using regularised least squares then the arbitrarily good readout layer $W$ \emph{will} be attained (with probability as close to 1 as desired). This result is analogous to the main result by \cite{Hart_regularised_2020} who prove a similar theorem for ESNs trained on deterministic inputs. Before we introduce the result we will present the definition of a stationary process, an ergodic process, and the ergodic theorem.

\begin{definition}
    (Stationary Process \cite{mcgoff2015}) A stochastic process $(\boldsymbol{Z}_k)_{k \in \mathbb{Z}} \equiv \boldsymbol{Z}$ is stationary if for any $\ell \in \mathbb{N}$ and finite subset $I \subset \mathbb{Z}$ the joint distribution $(\boldsymbol{Z}_i)_{i \in I}$ is equal to the joint distribution $(\boldsymbol{Z}_{i+\ell})_{i \in I}$.
\end{definition}

\begin{definition}
    (Stationary Ergodic Process \cite{mcgoff2015}) A stationary stochastic process $(\boldsymbol{Z}_k)_{k \in \mathbb{Z}} \equiv \boldsymbol{Z}$ is called ergodic if for every $\ell \in \mathbb{N}$ and every pair of Borel sets $A,B$
    \begin{align*}
        \lim_{\ell \to \infty} \frac{1}{\ell} \sum^{\ell-1}_{k=0} &\mathbb{P}\bigg( (\boldsymbol{Z}_1 , \ldots , \boldsymbol{Z}_{\ell}) \in A, (\boldsymbol{Z}_k , \ldots , \boldsymbol{Z}_{k+\ell}) \in B \bigg) \\
        = &\mathbb{P}\bigg( (\boldsymbol{Z}_1 , \ldots , \boldsymbol{Z}_{\ell}) \in A \bigg) \mathbb{P}\bigg( (\boldsymbol{Z}_1 , \ldots , \boldsymbol{Z}_{\ell}) \in B \bigg).
    \end{align*}
\end{definition}

Every stationary ergodic processes $Z$ satisfies the celebrated Ergodic Theorem.

\begin{theorem}
    (Ergodic Theorem) If $(\boldsymbol{Z}_k)_{k \in \mathbb{Z}} \equiv \boldsymbol{Z}$ is a stationary ergodic process then for any $i \in \mathbb{Z}$
    \begin{align*}
        \mathbb{E}_{\mu}[\boldsymbol{Z}_i] = \lim_{\ell \to \infty} \frac{1}{\ell}\sum_{k=0}^{\ell-1}  \boldsymbol{Z}_{i+k}
    \end{align*}
    almost surely.
\end{theorem}

Our result holds in the following supervised learning context. Given time series data $z_k$ (from an admissible, stationary, ergodic, bounded process $\boldsymbol{Z}$) and a time series of targets $r_k$ depending on all previous data $\ldots,z_{k-2}, z_{k-1}$ we wish to approximate the mapping from $\ldots,z_{k-2}, z_{k-1}$ to $r_k$. This mapping is denoted~$\mathcal{R}$. Our result states that an ESN with weights $\boldsymbol{A},\boldsymbol{C}$ and biases $\boldsymbol{\zeta}$ randomly generated by the procedure in definition~\ref{proc:procedure}, which is fed the training data $z_k$, and then trained by regularised least squares, will yield a matrix $W$. This ESN equipped with the matrix $W$ (denoted $H^{\boldsymbol{A},\boldsymbol{C},\boldsymbol{\zeta}}_W$) will approximate the relationship~$\mathcal{R}$ between data points $\ldots,z_{k-2}, z_{k-1}$ and targets $r_k$ as closely as required.

\begin{theorem}
    \label{generalised_training_theorem}
    Suppose that $\boldsymbol{Z}$ is an admissible input process, that is also stationary and ergodic, with invariant measure $\mu$. Let $\mathcal{R} : (D_n)^{\mathbb{Z}} \to \mathbb{R}$ (where $D_n$ is a compact subset of $\mathbb{R}^n$) be CTI, $\mu$-measurable, and satisfy $\mathbb{E}_{\mu}[|\mathcal{R}(\boldsymbol{Z})|^2] < \infty$. Let $z$ be an arbitrary realisation of $\boldsymbol{Z}$
    
    Then for any $\epsilon > 0$ and $\delta \in (0,1)$ there exist $N,T_0 \in \mathbb{N}$, $R>0$, $\lambda^* > 0$ and $\ell \in \mathbb{N}$ such that the ESN with parameters $\boldsymbol{A},\boldsymbol{C},\boldsymbol{\zeta}$ generated by the procedure in Definition~\ref{proc:procedure} (with inputs $N,T_0,R$), and $W^*_{\ell} \in \mathbb{R}^{2(d(T_0+1)+N)}$ which minimises (over $W\in \mathbb{R}^{2(d(T_0+1)+N)}$) the least squares problem
    \begin{align}
        \frac{1}{\ell}\sum_{k=0}^{\ell - 1} \left\lVert H_{W}^{\boldsymbol{A},\boldsymbol{C},\boldsymbol{\zeta}} T^{-k}(z) - \mathcal{R} T^{-k}(z) \right\rVert^2 + \lambda \left\lVert W \right\rVert^2, \nonumber
    \end{align}
    (where $\lambda \in (0,\lambda^*)$) satisfies
with probability $(1-\delta)$ the inequality
    \begin{align}
        \mathbb{E}_{\mu}\left[ \left. \left\lVert H_{W^*_{\ell}}^{\boldsymbol{A},\boldsymbol{C},\boldsymbol{\zeta}}(\boldsymbol{Z}) - \mathcal{R}(\boldsymbol{Z}) \right\rVert^2 \right| \boldsymbol{A},\boldsymbol{C},\boldsymbol{\zeta} \right] < \epsilon. \nonumber
    \end{align}
\end{theorem}

\begin{proof}
    Later in this paper, we state and prove a more general result (Theorem \ref{offline_q_learning}) which reduces to this result in the special case $\gamma = 0$.
\end{proof}

In summary, we have stated that for any $\epsilon>0$ and $\delta \in (0,1)$ there exists an ESN of dimension $n = 2(d(T_0+1)+N)$ with output layer $W$ trained by the Tikhonov-regularised least squares procedure against $\ell$ training points, whose output functional approximates the target arbitrarily closely with arbitrarily high probability. The theorem is (sadly) non constructive in the sense that the number of neurons $n$, number of training points $\ell$ and regularisation parameter $\lambda^*$ are not computed for a given $\epsilon$ and $\delta$. Ideally, we would establish uniform bounds on the number of number of neurons $n$ and data points $\ell$ required for an approximation with tolerance $\epsilon$ to hold with probability $\delta$. Though less ideal, one could establish an asymptotic order of convergence using the central limit theorem (CLT). The CLT (roughly) states that the error between the time average and the space average of a stationary ergodic process converges in law to a normal distribution with standard deviation of the order $1/\sqrt{\ell}$ as the number of data points $\ell$ grows to infinity. The CLT is stated below. 

\begin{theorem}
    (Central Limit Theorem \cite{mcgoff2015}) If $(\boldsymbol{Z}_k)_{k \in \mathbb{Z}}$ is a stationary ergodic process then there exists a covariance matrix $\Sigma$ such that for any $i \in \mathbb{Z}$ and Borel set $A$
    \begin{align*}
        \lim_{\ell \to \infty}\mathbb{P}\bigg(\frac{1}{\sqrt{\ell}}\sum_{k=0}^{\ell-1} (\boldsymbol{Z}_{i+k} - \mathbb{E}_{\mu}[\boldsymbol{Z}_i])\bigg) = \mathbb{P} \big( \mathcal{N}(0,\Sigma) \in A \big).
    \end{align*}
    In other words, the random variables
    \begin{align*}
        \frac{1}{\sqrt{\ell}}\sum_{k=0}^{\ell-1} (\boldsymbol{Z}_{i+k} - \mathbb{E}_{\mu}[\boldsymbol{Z}_i])
    \end{align*}
    converge in distribution to the multivariate normal $\mathcal{N}(0,\Sigma)$ as $\ell \to \infty$.
\end{theorem}

This suggests that the approximation of the target functional $\mathcal{R}$ also converges with order $1/\sqrt{\ell}$ as the number of data points increases. Furthermore, related results by \cite{Gonon2020} use the CLT to establish uniform bounds on the number of neurons $n = 2(d(T+1) + N)$ required for a given approximation. This strongly suggests that the approximation in Theorem \ref{generalised_training_theorem} converges with order $1/\sqrt{N}$. 

We will now pivot towards our second novel result, which generalises the first. Suppose that we have a contraction mapping $\Phi$ on the space of functionals, and we seek a $W^*$ such that the ESN functional $H_{W^*}^{\boldsymbol{A},\boldsymbol{C},\boldsymbol{\zeta}}$ approximates the unique fixed point of $\Phi$. The existence of the unique fixed point is guaranteed by Banach's fixed point theorem. Finding the fixed point of a contraction mapping has applications in reinforcement learning because the optimal value function (and optimal quality function) of a Markov Decision Process (MDP) is a fixed point of a \emph{Bellman operator}. The theory we are presenting here can be viewed as a generalisation of an MDP because the input processes we are considering may have long time correlations (violating the Markov property) which can only be recognised by filters with sufficiently long and robust memories; like Echo State Networks.

We can observe first of all if $\Phi$ is the constant map $\Phi(H) = \mathcal{R}$, then $\Phi$ is clearly a contraction mapping with fixed point $\mathcal{R}$. In this case, the problem is exactly the same as that solved by Theorem \ref{generalised_training_theorem}.
We are especially interested in the case of $\Phi$ taking the form of the \emph{Bellman Value operator}. To make this formal, we will consider a stationary ergodic process $\boldsymbol{Z}$ with invariant measure $\mu$. Then we define the map $T_{\boldsymbol{Z}}$ as a CTI filter on the bi-infinite sequences $(D_N)^{\mathbb{Z}}$, which returns the random variable:
\begin{align*}
    T_{\boldsymbol{Z}}(z)_k =
    \begin{cases}
        T_{\boldsymbol{Z}}(z)_{k+1} &\text{ if } k < 0 \\
        \boldsymbol{Z}_{k+1} \ | \ \boldsymbol{Z}_j = z_j \ \forall j \leq 0 &\text{ if } k \geq 0.
    \end{cases}
\end{align*}
Next, we introduce $\mathcal{R} : (D_N)^{\mathbb{Z}} \to \mathbb{R}$ as the CTI reward functional, giving a reward (or expectation over a distribution of rewards) to an agent that has observed a given sequence of (reward,  action, observation) triples. We let $\gamma \in [0,1)$ denote the discount factor, and define the operator
\begin{align}
    \Phi(H)(z) := \mathcal{R}(z) + \gamma \mathbb{E}_{\mu}[ HT_{\boldsymbol{Z}}(z)]. \label{Phi}
\end{align}

In this case, $\Phi$ is a contraction mapping with Lipschitz constant $\gamma$. With this, we will define the CTI value functional $V : (D_N)^{\mathbb{Z}} \to \mathbb{R}$ (with respect to the process $\boldsymbol{Z}$) as
\begin{align*}
    V(z) &:= \mathbb{E}_{\mu}\bigg[ \sum_{k=0}^{\infty} \gamma^k \mathcal{R}T^k(\boldsymbol{Z}) \ \bigg| \ \boldsymbol{Z}_j = z_j \ \forall j \leq 0 \bigg].
\end{align*}
The value functional $V$ takes a sequence of (reward, action, observation) triples and returns the expected discounted sum of future rewards.
Furthermore, the value function $V$ is the unique fixed point of the Bellman operator $\Phi$. Re-arranging the definition of $V(z)$ above, we have that:
\begin{align*}
    V(z) &= \mathbb{E}_{\mu}\bigg[ \sum_{k=0}^{\infty} \gamma^k \mathcal{R}T^k(\boldsymbol{Z}) \ \bigg| \ \boldsymbol{Z}_j = z_j \ \forall j \leq 0 \bigg] \\
    &= \mathbb{E}_{\mu}\bigg[ \sum_{k=1}^{\infty} \gamma^k \mathcal{R}T^k(\boldsymbol{Z}) \ \bigg| \ \boldsymbol{Z}_j = z_j \ \forall j \leq 0 \bigg] + \mathcal{R}(z) \\
    &= \gamma \mathbb{E}_{\mu}\bigg[ \sum_{k=0}^{\infty} \gamma^k \mathcal{R}T^{k+1}(\boldsymbol{Z}) \ \bigg| \ \boldsymbol{Z}_j = z_j \ \forall j \leq 0 \bigg] + \mathcal{R}(z) \\
    &= \gamma \mathbb{E}_{\mu}\bigg[ \sum_{k=0}^{\infty} \gamma^k \mathcal{R}T^k(\boldsymbol{Z}) \ \bigg| \ \boldsymbol{Z}_j = z_j \ \forall j < 0 \bigg] + \mathcal{R}(z) \\
\end{align*}
where we have carried out straightforward relabellings of the indexing of terms in the sum by $k$. Then by the law of total expectation we may write this last expression as
\begin{align*}    
    V(z) &=\gamma\mathbb{E}_{\mu}\bigg[\mathbb{E}_{\mu}\bigg[ \sum^{\infty}_{k=0}\gamma^k \mathcal{R}T^k(\boldsymbol{Z}) \ \bigg| \ \boldsymbol{Z}_j = T_{\boldsymbol{Z}}(z)_j \ \forall j \leq 0 \bigg]\bigg] + \mathcal{R}(z) \\
    &= \gamma \mathbb{E}_{\mu}[VT_{\boldsymbol{Z}}(z)] + \mathcal{R}(z)
    = \Phi(V)(z),
\end{align*}
which shows that $V$ is indeed a fixed point of $\Phi$, and so is the unique such, since $\Phi$ is a contraction.

Our goal is now to seek a $W^*$ such that the ESN functional $H^{\boldsymbol{A},\boldsymbol{C},\boldsymbol{\zeta}}_{W^*}$ closely approximates the unique fixed point $V$ of $\Phi$. One approach is to collect a dataset from a single training trajectory, and then perform least squares regression to find $W^*$. This is an example of \emph{offline} learning (in the reinforcement learning parlance) because the training occurs after the data has been collected. This is in contrast to \emph{online} learning where training takes place dynamically as new data becomes available. We will make this offline approach formal in the following theorem.

\begin{theorem}
\label{offline_q_learning}
    Suppose that $\boldsymbol{Z}$ is an admissible input process, that is also stationary and ergodic with invariant measure $\mu$. Let $\mathcal{R} : (D_N)^{\mathbb{Z}} \to \mathbb{R}$ be $\mu$-measurable and satisfy $\mathbb{E}[|\mathcal{R}(\boldsymbol{Z})|^2] < \infty$ and define $\Phi$ using \eqref{Phi}
    on the $\mu$-measurable functionals $H$ that satisfy $\mathbb{E}_{\mu}[|H(\boldsymbol{Z})|^2] < \infty$. Let $\gamma \in [0,1)$. Let $z$ be an arbitrary realisation of $\boldsymbol{Z}$
    
    Then for any $\epsilon > 0$, $\delta \in (0,1)$ there exists $N,T_0 \in \mathbb{N}, R, \lambda^* > 0$ and $\ell \in \mathbb{N}$ such that the ESN with parameters $\boldsymbol{A},\boldsymbol{C},\boldsymbol{\zeta}$ generated by procedure \ref{proc:procedure} (with inputs $N,T_0,R$), and $W^{*}_{\ell} \in \mathbb{R}^{2(d(T_0+1)+N)}$ minimising (over $W \in \mathbb{R}^{2(d(T_0+1)+N)}$) the least squares problem
    \begin{align}
        \frac{1}{\ell}\sum_{k = 0}^{\ell - 1} \left\lVert W^{\top} (H^{\boldsymbol{A},\boldsymbol{C},\boldsymbol{\zeta}}T^{-k}(z) - \gamma H^{\boldsymbol{A},\boldsymbol{C},\boldsymbol{\zeta}}T^{1-k}(z)) - \mathcal{R}(z) \right\rVert^2 + \lambda \lVert W \rVert^2 \nonumber
    \end{align}
    where $\lambda \in (0,\lambda^*)$, then with probability $(1-\delta)$
    \begin{align}
        \mathbb{E}_{\mu}\left[ \left. \left\lVert H_{W^*_{\ell}}^{\boldsymbol{A},\boldsymbol{C},\boldsymbol{\zeta}}(\boldsymbol{Z}) - \Phi H_{W^*_{\ell}}^{\boldsymbol{A},\boldsymbol{C},\boldsymbol{\zeta}}(\boldsymbol{Z}) \right\rVert^2 \right| \boldsymbol{A},\boldsymbol{C},\boldsymbol{\zeta} \right] < \epsilon. \nonumber
    \end{align}
\end{theorem}
\begin{proof}
    First let $V$ be the unique fixed point of the contraction mapping $\Phi$ whose existence and uniqueness is guaranteed by Banach's fixed point theorem. Denote the Lipschitz constant of $\Phi$ with the symbol $\tau$.
    Then we fix $\epsilon > 0$ and $\delta \in (0,1)$, then by Theorem \ref{general_ESN_approximation_theorem} there exists with probability $(1-\delta)$ a linear readout $W \in \mathbb{R}^{2(d(T_0+1) + N)}$ such that
    \begin{align}
        \mathbb{E}_{\mu}\left[ \left. \left\lVert H_W^{\boldsymbol{A},\boldsymbol{C},\boldsymbol{\zeta}}(\boldsymbol{Z}) - V(\boldsymbol{Z}) \right\rVert^{2} \ \right| \boldsymbol{A} , \boldsymbol{C} , \boldsymbol{\zeta} \right] < \frac{\epsilon}{5(1 + \tau)}. 
        \label{HW-H*}
    \end{align}
    Then it follows that
    \begin{align*}
        \mathbb{E}_{\mu}\left[ \left\lVert H^{\boldsymbol{A},\boldsymbol{C},\boldsymbol{\zeta}}_{W} - \Phi H^{\boldsymbol{A},\boldsymbol{C},\boldsymbol{\zeta}}_{W} \right\rVert^2 | \boldsymbol{A},\boldsymbol{C},\boldsymbol{\zeta} \right]  \\
         & \dawesspace = \mathbb{E}_{\mu}\left[ \left\lVert H^{\boldsymbol{A},\boldsymbol{C},\boldsymbol{\zeta}}_{W}(\boldsymbol{Z}) - \Phi H^{\boldsymbol{A},\boldsymbol{C},\boldsymbol{\zeta}}_{W}(\boldsymbol{Z}) + V(\boldsymbol{Z}) - V(\boldsymbol{Z}) \right\rVert^2 | \boldsymbol{A},\boldsymbol{C},\boldsymbol{\zeta} \right] \\
        &  \dawesspace \leq \mathbb{E}_{\mu}[ \lVert H^{\boldsymbol{A},\boldsymbol{C},\boldsymbol{\zeta}}_{W}(\boldsymbol{Z}) - V(\boldsymbol{Z}) \rVert^2 | \boldsymbol{A},\boldsymbol{C},\boldsymbol{\zeta} ] + \mathbb{E}_{\mu}[ \lVert V(\boldsymbol{Z}) - \Phi H^{\boldsymbol{A},\boldsymbol{C},\boldsymbol{\zeta}}_W(\boldsymbol{Z}) \rVert^2 | \boldsymbol{A},\boldsymbol{C},\boldsymbol{\zeta} ] \\
        &  \dawesspace = \mathbb{E}_{\mu}[ \lVert H^{\boldsymbol{A},\boldsymbol{C},\boldsymbol{\zeta}}_{W}(\boldsymbol{Z}) - V(\boldsymbol{Z}) \rVert^2 | \boldsymbol{A},\boldsymbol{C},\boldsymbol{\zeta} ] + \mathbb{E}_{\mu}[ \lVert \Phi V(\boldsymbol{Z}) - \Phi H^{\boldsymbol{A},\boldsymbol{C},\boldsymbol{\zeta}}_W(\boldsymbol{Z}) \rVert^2 | \boldsymbol{A},\boldsymbol{C},\boldsymbol{\zeta} ] \\
        &  \dawesspace \leq \mathbb{E}_{\mu}[ \lVert H^{\boldsymbol{A},\boldsymbol{C},\boldsymbol{\zeta}}_{W}(\boldsymbol{Z}) - V(\boldsymbol{Z}) \rVert^2 | \boldsymbol{A},\boldsymbol{C},\boldsymbol{\zeta} ] + \tau \mathbb{E}_{\mu}[ \lVert V(\boldsymbol{Z}) - H^{\boldsymbol{A},\boldsymbol{C},\boldsymbol{\zeta}}_W(\boldsymbol{Z}) \rVert^2 | \boldsymbol{A},\boldsymbol{C},\boldsymbol{\zeta} ] \\
        &  \dawesspace = (1 + \tau) \mathbb{E}_{\mu}[ \lVert V(\boldsymbol{Z}) - H^{\boldsymbol{A},\boldsymbol{C},\boldsymbol{\zeta}}_W(\boldsymbol{Z}) \rVert^2 | \boldsymbol{A},\boldsymbol{C},\boldsymbol{\zeta} ]\\
        &  \dawesspace < (1 + \tau) \frac{\epsilon}{5(1+\tau)} \text{ by \eqref{HW-H*}} \\
        &  \dawesspace < \frac{\epsilon}{5}
    \end{align*}
    which yields the estimate 
    \begin{align}
        \mathbb{E}_{\mu}[ \lVert H^{\boldsymbol{A},\boldsymbol{C},\boldsymbol{\zeta}}_{W} - \Phi H^{\boldsymbol{A},\boldsymbol{C},\boldsymbol{\zeta}}_{W} \rVert^2 | \boldsymbol{A},\boldsymbol{C},\boldsymbol{\zeta} ] < \frac{\epsilon}{5}. \label{H-PhiH}
    \end{align}
    Now, we can choose $\lambda^*$ such that for any $\lambda \in (0,\lambda^*)$
    \begin{align}
        \lambda \lVert W \rVert^2 < \frac{\epsilon}{5}.
        \label{W_lambda<eps}
    \end{align}
    Next we define a sequence of vectors $(W^{*}_j)_{j \in \mathbb{N}}$ by
     \begin{align*}
         W^*_{j} = \argmin_{U \in \mathbb{R}^{2(d(T_0+1) + N)}} \bigg( \frac{1}{j} \sum_{k=0}^{j - 1} \lVert H_{U}^{\boldsymbol{A}, \boldsymbol{C}, \boldsymbol{\zeta}} T^{-k}(z) - \gamma H^{\boldsymbol{A},\boldsymbol{C},\boldsymbol{\zeta}}_U T^{1-k}(z) - \mathcal{R}T^{-k}(z) \rVert^2 + \lambda \lVert U \rVert^2 \bigg).
     \end{align*}
     We may view $\argmin$ as continuous map on the space of strictly convex $C^1$ functions that returns their unique minimiser. The regularised linear least squares problem is a strictly convex $C^1$ problem, so we may define $W^*_{\infty} \in \mathbb{R}^{2d(T_0+1)+N}$ by
     \begin{align*}
        W^*_{\infty} &:= \argmin_{U} \bigg( \mathbb{E}_{\mu}[ \lVert H_{U}^{A, C, \zeta} (\boldsymbol{Z}) - \gamma H^{\boldsymbol{A},\boldsymbol{C},\boldsymbol{\zeta}}_U T(\boldsymbol{Z}) - \mathcal{R}(\boldsymbol{Z}) \rVert^2 | \boldsymbol{A},\boldsymbol{C},\boldsymbol{\zeta} ] + \lambda \lVert U \rVert^2 \bigg) \\
        &= \argmin_{U} \lim_{j \to \infty} \bigg( \frac{1}{j} \sum_{k=0}^{j - 1} \lVert H_{U}^{\boldsymbol{A}, \boldsymbol{C}, \boldsymbol{\zeta}} T^{-k}(z) - \gamma H^{\boldsymbol{A},\boldsymbol{C},\boldsymbol{\zeta}}_U T^{1-k} - \mathcal{R}T^{-k}(z) \rVert^2 + \lambda \lVert U \rVert^2 \bigg) \\
        &= \lim_{j \to \infty} \argmin_{U} \bigg( \frac{1}{j} \sum_{k=0}^{j - 1} \lVert H_{U}^{\boldsymbol{A}, \boldsymbol{C}, \boldsymbol{\zeta}} T^{-k}(z) - \gamma H^{\boldsymbol{A},\boldsymbol{C},\boldsymbol{\zeta}}_U T^{1-k} - \mathcal{R}T^{-k}(z) \rVert^2 + \lambda \lVert U \rVert^2 \bigg) \\ 
        &= \lim_{j \to \infty} W^*_j 
     \end{align*}
     where the second and third equalities hold by the Ergodic Theorem and continuity of $\argmin$ respectively.
    Now, we may choose $\ell \in \mathbb{N}$ sufficiently large that
    \begin{multline}
        \left| \mathbb{E}_{\mu}[\lVert W^{*\top}_{\ell} ( H^{\boldsymbol{A},\boldsymbol{C},\boldsymbol{\zeta}}(\boldsymbol{Z}) - \gamma H^{\boldsymbol{A},\boldsymbol{C},\boldsymbol{\zeta}}T(\boldsymbol{Z}) ) - \mathcal{R}(\boldsymbol{Z}) \rVert^2 | \boldsymbol{A},\boldsymbol{C},\boldsymbol{\zeta}] \right. \\
        \left. - \ \mathbb{E}_{\mu}[\lVert W^{*\top}_{\infty} ( H^{\boldsymbol{A},\boldsymbol{C},\boldsymbol{\zeta}}(\boldsymbol{Z}) - \gamma H^{\boldsymbol{A},\boldsymbol{C},\boldsymbol{\zeta}}T(\boldsymbol{Z}) ) - \mathcal{R}(\boldsymbol{Z}) \rVert^2 | \boldsymbol{A},\boldsymbol{C},\boldsymbol{\zeta}] \right| < \frac{\epsilon}{5},
        \label{EW_l-EW_inf<eps}
    \end{multline}
    and
    \begin{multline}
        \left| \lim_{j \to \infty} \bigg( \frac{1}{j} \sum_{k = 0}^{j-1} \lVert W^{*\top}_{j}(H^{\boldsymbol{A},\boldsymbol{C},\boldsymbol{\zeta}}T^{-k}(z) - \gamma H^{\boldsymbol{A},\boldsymbol{C},\boldsymbol{\zeta}}T^{1-k}(z) ) - \mathcal{R}T^{-k}(z) \rVert^2 + \lambda \lVert W^*_{j} \rVert^2 \bigg) \right. \\
        \left. - \ \frac{1}{\ell} \sum_{k = 0}^{\ell-1} \lVert W^{*\top}_{\ell}(H^{\boldsymbol{A},\boldsymbol{C},\boldsymbol{\zeta}}T^{-k}(z) - \gamma H^{\boldsymbol{A},\boldsymbol{C},\boldsymbol{\zeta}}T^{1-k}(z) ) - \mathcal{R}T^{-k}(z) \rVert^2 + \lambda \lVert W^*_{\ell} \rVert^2 \right| < \frac{\epsilon}{5},
        \label{W-W}
    \end{multline}
    and by the Ergodic Theorem
    \begin{multline}
        \left\lvert \frac{1}{\ell} \sum_{k = 0}^{\ell-1} \lVert W^\top (H^{\boldsymbol{A},\boldsymbol{C},\boldsymbol{\zeta}}T^{-k}(z) - \gamma H^{\boldsymbol{A},\boldsymbol{C},\boldsymbol{\zeta}}T^{1-k}(z)) - \mathcal{R}(z) \rVert^2 \right. \\
        \left. - \lim_{j \to \infty} \frac{1}{j} \sum_{k = 0}^{j-1} \lVert W^\top (H^{\boldsymbol{A},\boldsymbol{C},\boldsymbol{\zeta}}T^{-k}(z) - \gamma H^{\boldsymbol{A},\boldsymbol{C},\boldsymbol{\zeta}}T^{1-k}(z)) - \mathcal{R}(z) \rVert^2 \right\rvert
        < \frac{\epsilon}{5}.
        \label{l-liml<eps}
    \end{multline}
    Now the proof proceeds directly
\begin{align*}
     \mathbb{E}_{\mu}[ \lVert H_{W^*_{\ell}}^{\boldsymbol{A},\boldsymbol{C},\boldsymbol{\zeta}}(\boldsymbol{Z}) - \Phi H_{W^*_{\ell}}^{\boldsymbol{A},\boldsymbol{C},\boldsymbol{\zeta}}(\boldsymbol{Z}) \rVert^2 | \boldsymbol{A},\boldsymbol{C},\boldsymbol{\zeta} ] \nonumber \\
        &  \dawesspace = \mathbb{E}_{\mu}[ \lVert H_{W^*_{\ell}}^{\boldsymbol{A},\boldsymbol{C},\boldsymbol{\zeta}}(\boldsymbol{Z}) - \gamma H_{W^*_{\ell}}^{\boldsymbol{A},\boldsymbol{C},\boldsymbol{\zeta}}T(\boldsymbol{Z}) - \mathcal{R}(\boldsymbol{Z}) \rVert^2 | \boldsymbol{A},\boldsymbol{C},\boldsymbol{\zeta} ] \\
        & \dawesspace = \mathbb{E}_{\mu}[\lVert W^{*\top}_{\ell} ( H^{\boldsymbol{A},\boldsymbol{C},\boldsymbol{\zeta}}(\boldsymbol{Z}) - \gamma H^{\boldsymbol{A},\boldsymbol{C},\boldsymbol{\zeta}}T(\boldsymbol{Z}) ) - \mathcal{R}(\boldsymbol{Z}) \rVert^2 | \boldsymbol{A},\boldsymbol{C},\boldsymbol{\zeta}].
        \end{align*}
Then we apply~\eqref{EW_l-EW_inf<eps} which yields
\begin{align*}
    \mathbb{E}_{\mu}[ \lVert H_{W^*_{\ell}}^{\boldsymbol{A},\boldsymbol{C},\boldsymbol{\zeta}}(\boldsymbol{Z}) - \Phi H_{W^*_{\ell}}^{\boldsymbol{A},\boldsymbol{C},\boldsymbol{\zeta}}(\boldsymbol{Z}) \rVert^2 | \boldsymbol{A},\boldsymbol{C},\boldsymbol{\zeta} ] \\
    & \dawesspace < \mathbb{E}_{\mu}[\lVert W^{*\top}_{\infty} ( H^{\boldsymbol{A},\boldsymbol{C},\boldsymbol{\zeta}}(\boldsymbol{Z}) - \gamma H^{\boldsymbol{A},\boldsymbol{C},\boldsymbol{\zeta}}T(\boldsymbol{Z}) ) - \mathcal{R}(\boldsymbol{Z}) \rVert^2 | \boldsymbol{A},\boldsymbol{C},\boldsymbol{\zeta}] + \frac{\epsilon}{5}.
\end{align*}
Then we apply the Ergodic Theorem
\begin{align*}
    \mathbb{E}_{\mu}[\lVert W^{*\top}_{\infty} ( H^{\boldsymbol{A},\boldsymbol{C},\boldsymbol{\zeta}}(\boldsymbol{Z}) - \gamma H^{\boldsymbol{A},\boldsymbol{C},\boldsymbol{\zeta}}T(\boldsymbol{Z}) ) - \mathcal{R}(\boldsymbol{Z}) \rVert^2 | \boldsymbol{A},\boldsymbol{C},\boldsymbol{\zeta}] + \frac{\epsilon}{5} \nonumber \\
    & \dawesspace\dawesspace = \lim_{j \to \infty} \bigg( \frac{1}{j} \sum_{k = 0}^{j-1} \lVert W^{*\top}_{\infty}(H^{\boldsymbol{A},\boldsymbol{C},\boldsymbol{\zeta}}T^{-k}(z) - \gamma H^{\boldsymbol{A},\boldsymbol{C},\boldsymbol{\zeta}}T^{1-k}(z) ) - \mathcal{R}T^{-k}(z) \rVert^2 \bigg) + \frac{\epsilon}{5} \\
    & \dawesspace\dawesspace \leq \lim_{j \to \infty} \bigg( \frac{1}{j} \sum_{k = 0}^{j-1} \lVert W^{*\top}_{\infty}(H^{\boldsymbol{A},\boldsymbol{C},\boldsymbol{\zeta}}T^{-k}(z) - \gamma H^{\boldsymbol{A},\boldsymbol{C},\boldsymbol{\zeta}}T^{1-k}(z) ) - \mathcal{R}T^{-k}(z) \rVert^2 \bigg) + \lambda \lVert W^*_{\infty} \rVert^2 + \frac{\epsilon}{5} \\
    & \dawesspace\dawesspace = \lim_{j \to \infty} \bigg( \frac{1}{j} \sum_{k = 0}^{j-1} \lVert W^{*\top}_{j}(H^{\boldsymbol{A},\boldsymbol{C},\boldsymbol{\zeta}}T^{-k}(z) - \gamma H^{\boldsymbol{A},\boldsymbol{C},\boldsymbol{\zeta}}T^{1-k}(z) ) - \mathcal{R}T^{-k}(z) \rVert^2 + \lambda \lVert W^*_{j} \rVert^2 \bigg) + \frac{\epsilon}{5} \\
    &\text{ then apply \eqref{W-W}} \\
    & \dawesspace\dawesspace < \frac{1}{\ell} \sum_{k = 0}^{\ell-1} \lVert W^{*\top}_{\ell}(H^{\boldsymbol{A},\boldsymbol{C},\boldsymbol{\zeta}}T^{-k}(z) - \gamma H^{\boldsymbol{A},\boldsymbol{C},\boldsymbol{\zeta}}T^{1-k}(z) ) - \mathcal{R}T^{-k}(z) \rVert^2 + \lambda \lVert W^*_{\ell} \rVert^2 + \frac{2\epsilon}{5} \\
    & \dawesspace\dawesspace \leq \frac{1}{\ell} \sum_{k = 0}^{\ell-1} \lVert W^{\top}(H^{\boldsymbol{A},\boldsymbol{C},\boldsymbol{\zeta}}T^{-k}(z) - \gamma H^{\boldsymbol{A},\boldsymbol{C},\boldsymbol{\zeta}}T^{1-k}(z) ) - \mathcal{R}T^{-k}(z) \rVert^2 + \lambda \lVert W \rVert^2 + \frac{2\epsilon}{5} \\
    &\text{ then apply \eqref{l-liml<eps}} \\
    & \dawesspace\dawesspace < \lim_{j \to \infty} \bigg( \frac{1}{j} \sum_{k = 0}^{j-1} \lVert W^{\top}(H^{\boldsymbol{A},\boldsymbol{C},\boldsymbol{\zeta}}T^{-k}(z) - \gamma H^{\boldsymbol{A},\boldsymbol{C},\boldsymbol{\zeta}}T^{1-k}(z) ) - \mathcal{R}T^{-k}(z) \rVert^2 \bigg) + \lambda \lVert W \rVert^2 + \frac{3\epsilon}{5} \\
    &\text{ then apply \eqref{W_lambda<eps}} \\
    & \dawesspace\dawesspace < \lim_{j \to \infty} \bigg( \frac{1}{j} \sum_{k = 0}^{j-1} \lVert W^{\top}(H^{\boldsymbol{A},\boldsymbol{C},\boldsymbol{\zeta}}T^{-k}(z) - \gamma H^{\boldsymbol{A},\boldsymbol{C},\boldsymbol{\zeta}}T^{1-k}(z) ) - \mathcal{R}T^{-k}(z) \rVert^2 \bigg) + \frac{4\epsilon}{5} \\
    & \dawesspace\dawesspace \text{ Then apply the Ergodic Theorem again} \\
    & \dawesspace\dawesspace = \mathbb{E}_{\mu}[ \lVert W^{\top}(H^{\boldsymbol{A},\boldsymbol{C},\boldsymbol{\zeta}}(\boldsymbol{Z}) - \gamma H^{\boldsymbol{A},\boldsymbol{C},\boldsymbol{\zeta}}T(\boldsymbol{Z}) - \mathcal{R}(\boldsymbol{Z})) \rVert^2 | \boldsymbol{A},\boldsymbol{C},\boldsymbol{\zeta} ] + \frac{4\epsilon}{5} \\
    & \dawesspace\dawesspace = \mathbb{E}_{\mu}[ \lVert H^{\boldsymbol{A},\boldsymbol{C},\boldsymbol{\zeta}}_{W} - \Phi H^{\boldsymbol{A},\boldsymbol{C},\boldsymbol{\zeta}}_{W} \rVert^2 | \boldsymbol{A},\boldsymbol{C},\boldsymbol{\zeta} ]  + \frac{4\epsilon}{5} \\ 
    & \dawesspace\dawesspace \text{ then apply \eqref{H-PhiH}} \\
    & \dawesspace\dawesspace < \epsilon
    \end{align*}
\end{proof}

\subsection{Connection to Partially Observed Markov Decision Processes}

Theorem \ref{offline_q_learning} applies to a reinforcement learning scenario where the observations are a stationary and ergodic process. This includes the case where observations emerge from a partially observed, stationary and ergodic Markov decision process. These are themselves a special case of a partially observed Markov decision process (POMDP) which are a common scenario studied in the reinforcement learning community. In particular, the results in this paper apply to POMDPs in the special case that the underlying Markov process is stationary and ergodic. However, there exist stationary ergodic processes, which satisfy the conditions of Theorem \ref{offline_q_learning}, which are not the output of any partially observed decision Markov process.

The approach that we set out in this paper has a lot in common with POMDPs, but there are some subtle differences which we will clarify here. First of all, the value function in this paper is defined in terms of the complete sequence of (reward, action, observation) triples, rather than the current belief state. One advantage of our approach is that a belief state does not need to be computed explicitly, nor do any assumptions need to made about the relationship between the hidden state of the environment and the observations. In the setting of this paper, the reservoir states $x_k$ (which are explicitly computed by evaluating $H^{\boldsymbol{A},\boldsymbol{C},\boldsymbol{\zeta}}T^k(z)$ can be interpreted as latent states, very much like the latent states for POMDPs. We also stress that the value function $V$ and reservoir functionals $H^{\boldsymbol{A},\boldsymbol{C},\boldsymbol{\zeta}}$ and $H_{W^*_{\ell}}^{\boldsymbol{A},\boldsymbol{C},\boldsymbol{\zeta}}$ are causal and time invariant (CTI) so we are never using future information that is unavailable in the present, despite the input sequences being bi-infinite. Indeed, one of the strengths of our approach is that the learning procedure will be able to learn the impact of any unobserved or hidden states via the latent states $x_k$ and linear regression.

\subsection{Training ESNs with online learning}
\label{online_learning}

In some reinforcement learning applications, it is useful - or even essential - for the optimisation of $W$ to occur dynamically as new data comes in; such algorithms are called \emph{online} learning algorithms. In this section, we will present and discuss some preliminary novel results surrounding online learning algorithms that use ESNs. 
We will first introduce a lemma, stating that, under reasonable conditions, the ODE
    \begin{align}
        \frac{d}{dt}W = -h(W) := -\mathbb{E}_{\mu}\bigg[H^{\boldsymbol{A},\boldsymbol{C},\boldsymbol{\zeta}}(\boldsymbol{Z})\big( H^{\boldsymbol{A},\boldsymbol{C},\boldsymbol{\zeta}}_{W}(\boldsymbol{Z}) - \Phi H^{\boldsymbol{A},\boldsymbol{C},\boldsymbol{\zeta}}_{W}(\boldsymbol{Z}) \big)   \bigg]
        \label{associated_ODE}
    \end{align}
converges exponentially quickly to a globally asymptotic fixed point $W^*$, for which the associated ESN functional $H_{W^*}^{\boldsymbol{A},\boldsymbol{C},\boldsymbol{\zeta}}$ is \emph{close} to the unique fixed point of $\Phi$. By \emph{close} we mean that the orthogonal projection of $\Phi H_{W^*}^{\boldsymbol{A},\boldsymbol{C},\boldsymbol{\zeta}}$ onto the finite dimensional vector space of functionals $\{ H_{W}^{\boldsymbol{A},\boldsymbol{C},\boldsymbol{\zeta}} \ | \ W \in \mathbb{R}^d \}$ is $H_{W^*}^{\boldsymbol{A},\boldsymbol{C},\boldsymbol{\zeta}}$. Unlike the previous result (Theorem \ref{offline_q_learning}) we do not need to assume that the contraction mapping satisfies $\Phi(H) = R + \gamma \mathbb{E}[HT_{\boldsymbol{Z}}]$. We could choose for example $\Phi(H) = R + \gamma \sup_{\pi} \mathbb{E}[HT_{\boldsymbol{Z}(\pi)}]$ where $\boldsymbol{Z}(\pi)$ is a process under a control $\pi$. The fixed point of this operator is the \emph{optimal} value function $V^*$.

\begin{lemma}
    Let $\boldsymbol{Z}$ be an admissible input process. Let $\boldsymbol{A},\boldsymbol{C},\boldsymbol{\zeta}$ be a $n \times n$, $n \times d$, and $n \times 1$ dimensional random reservoir matrix, input matrix and bias vector. Let $H^{\boldsymbol{A},\boldsymbol{C},\boldsymbol{\zeta}}$ and $H_{W}^{\boldsymbol{A},\boldsymbol{C},\boldsymbol{\zeta}}$ denote the associated ESN functionals. Let $\Phi$ be a contraction mapping, with Lipschitz constant $0 \leq \tau < 1$, on the space of CTI filters $H : (D_N)^{\mathbb{Z}} \to \mathbb{R}$ that are $\mu$-measurable and satisfy $\mathbb{E}[H(\boldsymbol{Z})^2] < \infty$. Suppose further that $0 \leq \tau < \kappa^{-1} $ where $\kappa$ is the condition number of the autocorrelation matrix
    \begin{align*}
        \Sigma = \mathbb{E}_{\mu}\left[ \left. H^{\boldsymbol{A},\boldsymbol{C},\boldsymbol{\zeta}}(\boldsymbol{Z}) H^{\boldsymbol{A},\boldsymbol{C},\boldsymbol{\zeta} \top}(\boldsymbol{Z}) \ \right| \ \boldsymbol{A},\boldsymbol{C},\boldsymbol{\zeta} \right].
    \end{align*}
    Then there exists a $\delta > 0$ such that the ODE 
    \begin{align*}
        \frac{d}{dt}W = -h(W) := -\mathbb{E}_{\mu}\bigg[H^{\boldsymbol{A},\boldsymbol{C},\boldsymbol{\zeta}}(\boldsymbol{Z})\big( H^{\boldsymbol{A},\boldsymbol{C},\boldsymbol{\zeta}}_{W}(\boldsymbol{Z}) - \Phi H^{\boldsymbol{A},\boldsymbol{C},\boldsymbol{\zeta}}_{W}(\boldsymbol{Z}) \big) \ \bigg| \ \boldsymbol{A},\boldsymbol{C},\boldsymbol{\zeta}   \bigg]
    \end{align*}
    satisfies 
    \begin{align}
        \frac{d}{dt} \lVert W - W^* \rVert \leq - \delta \lVert W - W^* \rVert
        \label{exp_conv_cond}
    \end{align}
    where $W^*$ is a globally asymptotic fixed point. $W^*$ enjoys the further property that
    \begin{align*}
        H_{W^*}^{\boldsymbol{A},\boldsymbol{C},\boldsymbol{\zeta}} = \mathcal{P}\Phi H_{W^*}^{\boldsymbol{A},\boldsymbol{C},\boldsymbol{\zeta}}
    \end{align*}
    where $\mathcal{P}$ denotes the $L^2(\mu)$ orthogonal projection operator on the $\mu$-measurable filters $H$ satisfying $\mathbb{E}[H(\boldsymbol{Z})^2] < \infty$ and is defined
    \begin{align*}
        \mathcal{P}H(z) := H^{\boldsymbol{A},\boldsymbol{C},\boldsymbol{\zeta} \top}(z) \Sigma^{-1} \mathbb{E}_{\mu}\left[ \left. H^{\boldsymbol{A},\boldsymbol{C},\boldsymbol{\zeta}}(\boldsymbol{Z}) H(\boldsymbol{Z})\ \right| \ \boldsymbol{A},\boldsymbol{C},\boldsymbol{\zeta} \right].
    \end{align*}
\end{lemma}

\begin{proof}
    To show that $W^*$ is a globally asymptotic fixed point it suffices to show that there exists a $\delta > 0$ such that
     \begin{align*}
         (W - W^*) \cdot (h(W^*) - h(W)) \leq - \delta \lVert (W - W^*) \rVert ^2
     \end{align*}
     as this implies
     \begin{align*}
         \frac{d}{dt} \lVert W - W^* \rVert \leq - \delta \lVert W - W^* \rVert.
     \end{align*}
    To construct this $\delta$, we first note that
    \begin{align*}
        h(W) = \Sigma W - \mathbb{E}_{\mu} \bigg[ H^{\boldsymbol{A},\boldsymbol{C},\boldsymbol{\zeta}}(\boldsymbol{Z}) \Phi H^{\boldsymbol{A},\boldsymbol{C},\boldsymbol{\zeta}}_{W}(\boldsymbol{Z}) \big) \ \bigg| \ \boldsymbol{A},\boldsymbol{C},\boldsymbol{\zeta}  \bigg]
    \end{align*}
    so, by a direct computation we have
    \begin{align*}
         (W - W^*)  \cdot  (h(W^*) - h(W))  \\
         & \dawesspace = (W - W^*) \cdot \bigg( \mathbb{E}_{\mu} \bigg[ H^{\boldsymbol{A},\boldsymbol{C},\boldsymbol{\zeta}}(\boldsymbol{Z}) \Phi H^{\boldsymbol{A},\boldsymbol{C},\boldsymbol{\zeta}}_{W}(\boldsymbol{Z}) \big) \ \bigg| \ \boldsymbol{A},\boldsymbol{C},\boldsymbol{\zeta} \bigg] - \mathbb{E}_{\mu} \bigg[ H^{\boldsymbol{A},\boldsymbol{C},\boldsymbol{\zeta}}(\boldsymbol{Z}) \Phi H^{\boldsymbol{A},\boldsymbol{C},\boldsymbol{\zeta}}_{W^*}(\boldsymbol{Z}) \big) \ \bigg| \ \boldsymbol{A},\boldsymbol{C},\boldsymbol{\zeta} \bigg] \bigg) \\
        & \dawesspace - (W - W^*) \cdot \bigg( \Sigma W - \Sigma W^* \bigg) \\
        & \dawesspace =  (W - W^*) \cdot \bigg( \mathbb{E}_{\mu} \bigg[ H^{\boldsymbol{A},\boldsymbol{C},\boldsymbol{\zeta}}(\boldsymbol{Z}) \Phi H^{\boldsymbol{A},\boldsymbol{C},\boldsymbol{\zeta}}_{W}(\boldsymbol{Z}) \big) \ \bigg| \ \boldsymbol{A},\boldsymbol{C},\boldsymbol{\zeta} \bigg] - \mathbb{E}_{\mu} \bigg[ H^{\boldsymbol{A},\boldsymbol{C},\boldsymbol{\zeta}}(\boldsymbol{Z}) \Phi H^{\boldsymbol{A},\boldsymbol{C},\boldsymbol{\zeta}}_{W^*}(\boldsymbol{Z}) \big) \ \bigg| \ \boldsymbol{A},\boldsymbol{C},\boldsymbol{\zeta} \bigg] \bigg) \\
        & \dawesspace - (W - W^*)^{\top} \Sigma \big( W - W^* \big) \\
        & \dawesspace \leq
        (W - W^*) \cdot \bigg( \mathbb{E}_{\mu} \bigg[ H^{\boldsymbol{A},\boldsymbol{C},\boldsymbol{\zeta}}(\boldsymbol{Z}) \Phi H^{\boldsymbol{A},\boldsymbol{C},\boldsymbol{\zeta}}_{W}(\boldsymbol{Z}) \big) \ \bigg| \ \boldsymbol{A},\boldsymbol{C},\boldsymbol{\zeta} \bigg] - \mathbb{E}_{\mu} \bigg[ H^{\boldsymbol{A},\boldsymbol{C},\boldsymbol{\zeta}}(\boldsymbol{Z}) \Phi H^{\boldsymbol{A},\boldsymbol{C},\boldsymbol{\zeta}}_{W^*}(\boldsymbol{Z}) \big) \ \bigg| \ \boldsymbol{A},\boldsymbol{C},\boldsymbol{\zeta} \bigg] \bigg) \\
        & \dawesspace - \sigma \lVert W - W^* \rVert^2 \ \text{ where $\sigma$ is the smallest eigenvalue of $\Sigma$} \\
        &\dawesspace = 
        (W - W^*) \cdot \bigg( \mathbb{E}_{\mu} \bigg[ H^{\boldsymbol{A},\boldsymbol{C},\boldsymbol{\zeta}}(\boldsymbol{Z}) \Phi H^{\boldsymbol{A},\boldsymbol{C},\boldsymbol{\zeta}}_{W}(\boldsymbol{Z}) \big) - H^{\boldsymbol{A},\boldsymbol{C},\boldsymbol{\zeta}}(\boldsymbol{Z}) \Phi H^{\boldsymbol{A},\boldsymbol{C},\boldsymbol{\zeta}}_{W^*}(\boldsymbol{Z}) \big) \ \bigg| \ \boldsymbol{A},\boldsymbol{C},\boldsymbol{\zeta} \bigg] \bigg) \\
        & \dawesspace - \sigma \lVert W - W^* \rVert^2 \\
        & \dawesspace \leq 
        (W - W^*) \cdot \bigg( \mathbb{E}_{\mu} \bigg[ H^{\boldsymbol{A},\boldsymbol{C},\boldsymbol{\zeta}}(\boldsymbol{Z}) H^{\boldsymbol{A},\boldsymbol{C},\boldsymbol{\zeta}}_{W}(\boldsymbol{Z}) \big) - H^{\boldsymbol{A},\boldsymbol{C},\boldsymbol{\zeta}}(\boldsymbol{Z}) H^{\boldsymbol{A},\boldsymbol{C},\boldsymbol{\zeta}}_{W^*}(\boldsymbol{Z}) \big) \ \bigg| \ \boldsymbol{A},\boldsymbol{C},\boldsymbol{\zeta} \bigg] \bigg) \tau \\
        & \dawesspace - \sigma \lVert W - W^* \rVert^2 \ \text{ because $\tau$ is the Lipschitz constant for $\Phi$} \\
        & \dawesspace = \tau (W - W^*)^{\top} \Sigma (W - W^*) - \sigma \lVert W - W^* \rVert^2 \\
        & \dawesspace \leq \tau \rho \lVert W - W^* \rVert^2 - \sigma \lVert W - W^* \rVert^2 \ \text{ where $\rho$ is the largest eigenvalue of $\Sigma$} \\
        & \dawesspace = -(\sigma - \tau \rho) \lVert W - W^* \rVert^2,
    \end{align*}
    so we can set $\delta := \sigma - \tau \rho$ and notice $\delta > 0$ because $0 \leq \tau < \kappa^{-1} = \sigma / \rho$. Next, to show that
    \begin{align*}
        H_{W^*}^{\boldsymbol{A},\boldsymbol{C},\boldsymbol{\zeta}} = \mathcal{P}\Phi H_{W^*}^{\boldsymbol{A},\boldsymbol{C},\boldsymbol{\zeta}}
    \end{align*}
    we observe that since $W^*$ is an equilibrium point of the ODE 
    \begin{align*}
        \dot{W} = -h(W)
    \end{align*} 
    it follows that $h(W^*)=0$ and therefore
    \begin{align*}
    0 &= \mathbb{E}_{\mu}\bigg[H^{\boldsymbol{A},\boldsymbol{C},\boldsymbol{\zeta}}(\boldsymbol{Z})\big( H^{\boldsymbol{A},\boldsymbol{C},\boldsymbol{\zeta}}_{W^*}(\boldsymbol{Z}) - \Phi H^{\boldsymbol{A},\boldsymbol{C},\boldsymbol{\zeta}}_{W^*}(\boldsymbol{Z}) \big) \ \bigg| \ \boldsymbol{A},\boldsymbol{C},\boldsymbol{\zeta} \bigg] \\
    \implies 0 &= \mathbb{E}_{\mu}\bigg[H^{\boldsymbol{A},\boldsymbol{C},\boldsymbol{\zeta}}(\boldsymbol{Z}) H^{\boldsymbol{A},\boldsymbol{C},\boldsymbol{\zeta} \top}(\boldsymbol{Z}) \ \bigg| \ \boldsymbol{A},\boldsymbol{C},\boldsymbol{\zeta} \bigg] W^* - \mathbb{E}_{\mu}\bigg[ H^{\boldsymbol{A},\boldsymbol{C},\boldsymbol{\zeta}}(\boldsymbol{Z}) \Phi H^{\boldsymbol{A},\boldsymbol{C},\boldsymbol{\zeta}}_{W^*}(\boldsymbol{Z}) \ \bigg| \ \boldsymbol{A},\boldsymbol{C},\boldsymbol{\zeta} \bigg] \\
    \implies 0 &= \Sigma W^* - \mathbb{E}_{\mu}\bigg[ H^{\boldsymbol{A},\boldsymbol{C},\boldsymbol{\zeta}}(\boldsymbol{Z}) \Phi H^{\boldsymbol{A},\boldsymbol{C},\boldsymbol{\zeta}}_{W^*}(\boldsymbol{Z}) \ \bigg| \ \boldsymbol{A},\boldsymbol{C},\boldsymbol{\zeta} \bigg] \\
    \text{so, } \ \Sigma W^* 
    &= \mathbb{E}_{\mu}\bigg[ H^{\boldsymbol{A},\boldsymbol{C},\boldsymbol{\zeta}}(\boldsymbol{Z}) \Phi H^{\boldsymbol{A},\boldsymbol{C},\boldsymbol{\zeta}}_{W^*}(\boldsymbol{Z}) \ \bigg| \ \boldsymbol{A},\boldsymbol{C},\boldsymbol{\zeta} \bigg] \\
          \text{so, } \ W^* 
          &= \Sigma^{-1} \mathbb{E}_{\mu}\bigg[ H^{\boldsymbol{A},\boldsymbol{C},\boldsymbol{\zeta}}(\boldsymbol{Z}) \Phi H^{\boldsymbol{A},\boldsymbol{C},\boldsymbol{\zeta}}_{W^*}(\boldsymbol{Z}) \ \bigg| \ \boldsymbol{A},\boldsymbol{C},\boldsymbol{\zeta} \bigg] \\
          \text{so, } \ H^{\boldsymbol{A},\boldsymbol{C},\boldsymbol{\zeta}}_{W^*} 
          &= H^{\boldsymbol{A},\boldsymbol{C},\boldsymbol{\zeta} \top} \Sigma^{-1} \mathbb{E}_{\mu}\bigg[ H^{\boldsymbol{A},\boldsymbol{C},\boldsymbol{\zeta}}(\boldsymbol{Z}) \Phi H^{\boldsymbol{A},\boldsymbol{C},\boldsymbol{\zeta}}_{W^*}(\boldsymbol{Z}) \ \bigg| \ \boldsymbol{A},\boldsymbol{C},\boldsymbol{\zeta} \bigg] \\
          &= \mathcal{P}\Phi(H^{\boldsymbol{A},\boldsymbol{C},\boldsymbol{\zeta}}_{W^*}). \\
     \end{align*}
\end{proof}

One rather restrictive condition of this lemma is that the Lipschitz constant $\tau$ of the contraction $\Phi$ must be less than the reciprocal condition number $\kappa^{-1}$. $\kappa$ is a measure of \emph{how orthonormal} the columns of the autocorrelation matrix $\Sigma$ are. In particular, if the columns are indeed orthonormal, then $\kappa = 1$ and this condition ceases to be restrictive at all. If the columns are close to being linearly dependant, then $\kappa$ is large so the requirement that $\kappa^{-1}$ is small becomes troublesome. If indeed there is a linear dependence, the matrix $\Sigma$ is not even invertible and the theorem breaks down completely. If we interpret $H^{\boldsymbol{A},\boldsymbol{C},\boldsymbol{\zeta}}(\boldsymbol{Z})$ as a vector of features, then $\kappa$ grows with the correlation between features. Higher correlation between the features imposes a greater constraint on the Lipschitz constant $\tau$. If we have no inter-feature correlation then $\kappa = 1$ and we have no restriction at all on $\tau$.

To actually solve ODE \eqref{associated_ODE} we may need to compute
\begin{align}
    h(W) := \mathbb{E}_{\mu}\bigg[H^{\boldsymbol{A},\boldsymbol{C},\boldsymbol{\zeta}}(\boldsymbol{Z})\big( H^{\boldsymbol{A},\boldsymbol{C},\boldsymbol{\zeta}}_{W}(\boldsymbol{Z}) - \Phi H^{\boldsymbol{A},\boldsymbol{C},\boldsymbol{\zeta}}_{W}(\boldsymbol{Z}) \big) \ \bigg| \ \boldsymbol{A},\boldsymbol{C},\boldsymbol{\zeta}  \bigg] \label{h_definition}
\end{align}
which may, or may not, be practical. For example, if the process $\boldsymbol{Z}$ is ergodic, we can approximate \eqref{h_definition} by taking a sufficiently long time average of 
\begin{align*}
    H^{\boldsymbol{A},\boldsymbol{C},\boldsymbol{\zeta}} T^{k}(z)\big( H^{\boldsymbol{A},\boldsymbol{C},\boldsymbol{\zeta}}_{W} T^{k}(z) - \Phi H^{\boldsymbol{A},\boldsymbol{C},\boldsymbol{\zeta}}_{W} T^{k}(z) \big).
\end{align*}
Alternatively, we may approach the problem of solving \eqref{associated_ODE} by first considering the explicit Euler method (with time-steps $\alpha_k > 0$)
\begin{align*}
    W_{k+1} &= W_k - \alpha_k h(W_k) \\
            &= W_k - \alpha_k \mathbb{E}_{\mu}\bigg[H^{\boldsymbol{A},\boldsymbol{C},\boldsymbol{\zeta}}(\boldsymbol{Z})\big( H^{\boldsymbol{A},\boldsymbol{C},\boldsymbol{\zeta}}_{W_k}(\boldsymbol{Z}) - \Phi H^{\boldsymbol{A},\boldsymbol{C},\boldsymbol{\zeta}}_{W_k}(\boldsymbol{Z}) \big) \ \bigg| \ \boldsymbol{A},\boldsymbol{C},\boldsymbol{\zeta}  \bigg],
\end{align*}
    then we might (heuristically) expect the algorithm
\begin{align}
    W_{k+1} = W_k - \alpha_k H^{\boldsymbol{A},\boldsymbol{C},\boldsymbol{\zeta}} T^{k}(z)\big( H^{\boldsymbol{A},\boldsymbol{C},\boldsymbol{\zeta}}_{W_k} T^{k}(z) - \Phi H^{\boldsymbol{A},\boldsymbol{C},\boldsymbol{\zeta}}_{W_k} T^{k}(z) \big) 
    \label{stochastic_algorithm}
\end{align}
to converge to $W^*$, where $\alpha_k$ are positive definite real numbers that satisfy
\begin{align*}
    \sum_{k=1}^{\infty} \alpha_k = \infty \qquad \sum_{k=1}^{\infty} \alpha_k^2 < \infty.
\end{align*} 
We believe this heuristic could be made rigorous under mild assumptions, because algorithm \eqref{stochastic_algorithm} closely resembles the major algorithm extensively studied in \citep{Adaptive_Algorithms_and_Stochastic_Approximations} and \citep{Borkar2009}
for which similar results hold. Theorems 17 and 2.1.1. appearing in \citep{Adaptive_Algorithms_and_Stochastic_Approximations} and \citep{Borkar2009} respectively suggest that an algorithm much like \eqref{stochastic_algorithm} converges almost surely to $W^*$ if its associated ODE (reminiscent of  \eqref{associated_ODE}) satisfies condition \eqref{exp_conv_cond}, and the input process $\boldsymbol{Z}$ is strongly mixing. The conjecture that algorithm \eqref{stochastic_algorithm} converges to $W^*$ is also reminiscent of Theorem 3.1 by \cite{10.1007/978-3-540-72927-3_23}, and related results by \cite{Chen2019}. These results are closely related to Q-learning and stochastic gradient descent. We note that (sadly) finding the fixed point of the general contraction mapping $\Phi$ renders the estimation of $W$ a nonlinear problem. 

The theory yields an online reinforcement learning algorithm which we state below. We envision that the agent chooses a fixed policy $\pi$ and continues executing the policy for $\ell-1$ time steps. Under this policy, the agent makes observations $z_k$ and receives rewards $r_k$. We define $z_k(a)$ as the input to the ESN at time $k$ if the agent had instead executed action $a \in \mathcal{A}$ at time $k$.

\begin{algorithm}
\caption{Online Learning}
\label{online_algorithm}
\begin{algorithmic}[1]
\State Choose initial output layer $W_0$ and reservoir state $x_0$
\State Randomly generate $\boldsymbol{A},\boldsymbol{C},\boldsymbol{\zeta}$ according to procedure \ref{proc:procedure}
\State \textbf{for} each $k$ from $0$ to $\ell-1$
\Indent 
    \State Compute $W_{k+1} = W_k - \alpha_k x_k\big( W_k^{\top}x_k - r_k - \max_a \{ W_k^{\top} \sigma(\boldsymbol{A}x_k + \boldsymbol{C}z_k(a) + \boldsymbol{\zeta})\}\big)$
    \State Compute $x_{k+1} = \sigma(\boldsymbol{A}x_k + \boldsymbol{C}z_k + \boldsymbol{\zeta})$
\EndIndent
\end{algorithmic}
\end{algorithm}

\section{Bee World}
\label{Bee_World}

To demonstrate the theory presented in section~\ref{training_ESNs_with_least_squares}, we created a game called \emph{Bee World} and show that a simple reinforcement learning algorithm supported by an ESN can learn to play Bee World with respectable skill. The game is designed so that the theory presented previously is easy to visualise, rather than because the game is hard to master.

Bee World is set on the circle of unit circumference, which we denote by $S^1$, and represent as an interval with edges identified. At every point $y$ on the circle, there is a non-negative quantity of nectar which may be enjoyed by the bee without depletion. `Without depletion' means that the bee takes a negligible amount of nectar from the point $y$, so the bee occupying point $y$ does not cause the amount of nectar at $y$ to change. Furthermore, the nectar at every point $y$ varies with time $t$ according to the prescribed function
\begin{align}
    n(y,t) = 1 + \cos(\omega t) \sin(2 \pi y) \label{eqn:nectar}
\end{align}
(which we chose somewhat arbitrarily) that is unknown to the bee. Thus, the amount of nectar enjoyed by the bee at time $t$ is a value that lies in the interval $[0,2]$, which we will denote $\mathcal{N}$.
Time advances in discrete integer steps $t = 0,1,2,\ldots $, and at any time point $t$ a bee at point $y$ observes the quantity of nectar $r \in \mathcal{N}$ at point $y$ and nothing else. 
Having made this observation, the bee may choose to move anywhere in the interval $(y-c,y+c)$ for some fixed $0<c<1$ and arrive at its chosen destination at time $t+1$. The interval of possible moves $(-c,c)$ is called the action space and is denoted $\mathcal{A}$. The goal of the bee is to devise a policy whereby, given all its previous observations, the bee makes a decision as to where to move next, such that the discounted sum over all future nectar is as great as possible. The space of all previous (reward, action) pairs $(\mathcal{N} \times \mathcal{A})^{\mathbb{Z}_-}$ is contained by the space of bi-infinite sequences $\mathbb(\mathbb{R}^2)^{\mathbb{Z}}$. The agent playing Bee World makes no observations beyond the rewards (nectar) and actions, but we could easily envision a more general game where the agent makes observations from a set $\Omega$ and therefore makes its decisions based on a left sequence of (reward, action, observation) triples.

The policy adopted by the bee may be realised as a deterministic policy $\pi : (\mathcal{N} \times \mathcal{A})^{\mathbb{Z}} \to \mathcal{A}$ (a CTI functional) for which the bee executes an action $a \in \mathcal{A}$ determined by the history of (reward, action) pairs. Alternatively, the bee may adopt a stochastic policy, for which every state history of (reward, action) pairs admits a distribution over actions $\mathcal{A}$ from which the bee makes a random choice.

Though the evolution of Bee World is Markovian (and deterministic), the bee makes only a partial observation of the state of Bee World (i.e the amount of nectar the bee observes at time $t$) so the bee must take advantage of its \emph{memory} to reconstruct the true state and find an optimal policy. This need for memory renders the problem suitable for an ESN, while ruling out the conventional theory of Markov Decision Processes. The problem of playing Bee World can therefore be formulated as a Partially Observed Markov Decision Process.

\subsection{Approximating the value functional}

Under a policy $\pi$, the nectar-action pairs experienced by the bee yield a realisation of the $(\mathcal{N},\mathcal{A})^{\mathbb{Z}}$-valued random variable $\boldsymbol{Z}$. It therefore makes sense to define the value functional $V : (\mathcal{N} \times \mathcal{A})^{\mathbb{Z}} \to \mathbb{R}$ associated to $\boldsymbol{Z}$ by
\begin{align}
    V(z) = \mathbb{E}_{\mu}\bigg[ \sum_{k=0}^{\infty} \gamma^k \mathcal{R}T^k(\boldsymbol{Z}) \ \bigg| \ \boldsymbol{Z}_j = z_j \ \forall j \leq 0 \bigg] \label{bee_value}
\end{align}
 where $\mathcal{R} : (\mathcal{N} \times \mathcal{A})^{\mathbb{Z}} \to \mathbb{R}$ is the reward functional defined by $\mathcal{R}(z) = r_0$, where $r_0$ is the nectar collected at time $0$, $T$ is the shift operator, 
 and $\gamma \in [0,1)$ is the discount factor representing the relative importance of near and long term nectar consumption.
 We can see after a simple rearrangement of \eqref{bee_value}  that 
\begin{align*}
    V(z) = \mathcal{R}(z) + \gamma \mathbb{E}_{\mu}[VT_{\boldsymbol{Z}}(z)]
\end{align*}
so $V$ is the unique fixed point of the contraction mapping $\Phi$ defined by
\begin{align*}
    \Phi (H)(z) := \mathcal{R}(z) + \gamma \mathbb{E}_{\mu}[HT_{\boldsymbol{Z}}(z)]
\end{align*}
as discussed in Section \ref{training_ESNs_with_least_squares}. Thus, by Theorem \ref{offline_q_learning}, we can approximate the value function $V$ using an ESN trained by regularised least squares as long as the nectar-action pairs $z \in (\mathcal{N} \times \mathcal{A})^{\mathbb{Z}}$ are drawn from a suitable ergodic process $\boldsymbol{Z}$. Therefore, we chose an initial policy $\pi_0$ such that $\boldsymbol{Z}$ is ergodic. In particular, we chose a stochastic policy $\pi_0(z) \sim U(-c,c)$ for all histories of (reward, action) pairs $z \in (\mathcal{N} \times \mathcal{A})^{\mathbb{Z}}$ so that the bee takes a uniform sample from the action space $\mathcal{A} = (-c,c)$ at any point $y \in S^1$. For the purpose of playing a game, we set $c = 0.1$ and $\gamma = 0.5$. We allowed the bee to execute this policy for 2000 time steps and recorded the observed nectar at every time. The first 250 time steps are plotted in Figure \ref{fig:untrained_bee_main}.

Next, we set up an ESN of dimension $n = 300$, with reservoir matrix, input matrix, and bias $\boldsymbol{A},\boldsymbol{C},\boldsymbol{\zeta}$ populated with i.i.d uniform random variables $U(-0.05,0.05)$. $\boldsymbol{A}$ was then multiplied by a scaling factor such that the 2-norm of $\boldsymbol{A}$ satisfies $\lVert \boldsymbol{A} \rVert_2 = 1$. We choose an activation function $\sigma(x) := \max(0,x)$. We should pause here and note that ESN described here differs slightly from the ESN described in procedure \ref{proc:procedure}. We instead generated $\boldsymbol{A},\boldsymbol{C},\boldsymbol{\zeta}$ in a traditional way, which is empirically observed to be highly successful, as demonstrated in the literature, rather than the more cumbersome method described in procedure \ref{proc:procedure}. These numerical results suggest that procedure \ref{proc:procedure} can be simplified.

We then computed a sequence of reservoir states $x_k \in \mathbb{R}^{300}$ for the ESN using the iteration
\begin{align*}
    x_{k+1} = \sigma(\boldsymbol{A}x_k + \boldsymbol{C}z_k + \boldsymbol{\zeta})
\end{align*}
where $x_0 = 0$ and each $z_k \in (\mathcal{N} \times \mathcal{A})$ comprises 2 components: the first is the quantity of nectar observed by the bee at time $k$, and the second is the action $a \in (-c,c)$ executed at time $k$ under policy $\pi_0$. Now we return our attention to Theorem \ref{offline_q_learning}, and see that the $W^*_{\ell}$ minimising (over $W$)
\begin{align*}
    \frac{1}{\ell}\sum_{k = 0}^{\ell - 1} \lVert W^{\top} (H^{\boldsymbol{A},\boldsymbol{C},\boldsymbol{\zeta}}T^{-k}(z) - \gamma H^{\boldsymbol{A},\boldsymbol{C},\boldsymbol{\zeta}}T^{1-k}(z)) - \mathcal{R}(z) \rVert^2 + \lambda \lVert W \rVert^2
\end{align*}
converges to $W$ minimising
\begin{align}
    \lVert W^{\top} ( x_k - \gamma x_{k+1}) - r_k  \rVert^2 + \lambda \lVert W \rVert^2 \label{bee_world_lsq}
\end{align}
so we can immediately reformulate \eqref{bee_world_lsq} as the least squares problem
\begin{align*}
    W = (\Xi^\top \Xi + \lambda I)^{-1} \Xi^\top U  
\end{align*}
where $\Xi$ is the matrix with $k$th column
\begin{align*}
    \Xi_k := x_{k} - \gamma x_{k+1}
\end{align*}
and $U$ has $k$th entry $r_k$ the $k$th quantity of nectar, and $\lambda$ is the regularisation parameter which we set to $10^{-9}$.
We solved this linear system using the SVD. Now
\begin{align*}
    V(z) \approx H_{W^*_{\ell}}^{\boldsymbol{A},\boldsymbol{C},\boldsymbol{\zeta}}(z) \equiv (W^*_{\ell})^{\top} H^{\boldsymbol{A},\boldsymbol{C},\boldsymbol{\zeta}}(z) \equiv W^{\top}x
\end{align*}
where $x$ is the reservoir state associated to the left infinite input sequence $z$. Furthermore, the map $(W^{\top} \cdot)$ therefore approximates the unique fixed point of $\Phi$ (by Theorem \ref{offline_q_learning}) and this fixed point is exactly the value functional we are looking for. Thus, we can easily compute the approximate value of an arbitrary reservoir state $x$ under the initial policy $\pi$ by computing the inner product $W^{\top} x$. We illustrate this in Figure \ref{fig:untrained_bee_plot} by plotting, at each time $k = 1, \ldots  , 250$, the value of every observed state to accompany the observed nectar.

\begin{figure}[t!]
    \centering
    \begin{subfigure}[b]{1\textwidth}
        \includegraphics[width=\textwidth]{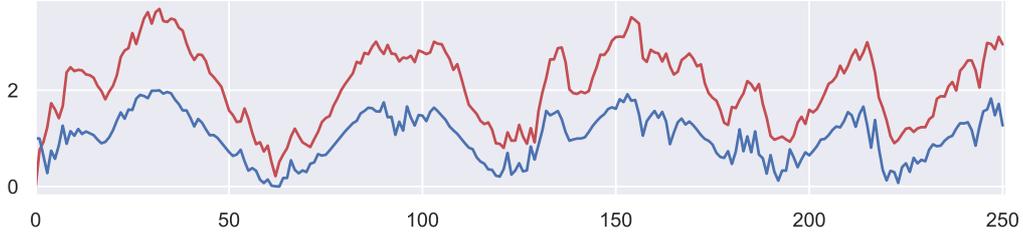}
        \caption{The nectar collected (blue) and the approximate value function under the initial policy $\pi_0$ (red) is plotted for the first 250 time steps ($x$-axis).}
        \label{fig:untrained_bee_plot}
    \end{subfigure}

    \begin{subfigure}[b]{1\textwidth}
        \includegraphics[width=1\textwidth]{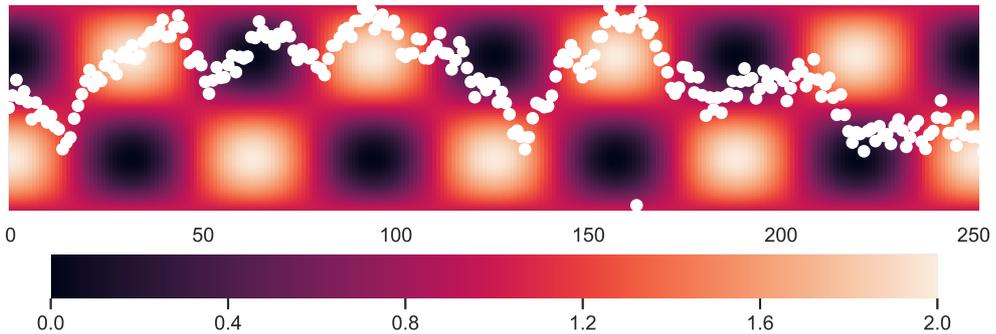}
        \caption{The nectar function $n(y,t)$~\eqref{eqn:nectar} at every point represented as a heat map in the $(t,y)$ plane, with the position of the bee at time $t$ under the initial policy indicated by the overlaid white circles.}
        \label{fig:untrained_bee}
    \end{subfigure}
\caption{Dynamics of Bee World where the bee executes the initial policy $\pi_0(z) \sim U(-0.1,0.1)$ for the first $250$ time steps.}
\label{fig:untrained_bee_main}
\end{figure}

\subsection{Updating the policy}
Having computed an approximate value function under the initial policy $\pi_0(z) \sim U(-0.1,0.1)$, we were faced with the problem of how to improve upon this policy. Exploring efficient and effective algorithms for iteratively improving a policy is a rich area of reinforcement learning research, but outside the scope of this section. Instead, we implemented a simple and greedy approach. For a given reservoir state $x$ we consider 100 actions $a_1, a_2, \ldots  a_{100}$ uniformly sampled over $\mathcal{A} = (-0.1,0.1)$, then for each action we consider the nectar-action pairs $z^{(1)} , \ldots  , z^{(100)} \in \mathcal{N} \times \mathcal{A}$ where the nectar for each pair is the current nectar; and is therefore the same in every pair. Then we compute the next reservoir states for each pair
\begin{align*}
    x^{(i)}_{k+1} = \sigma(\boldsymbol{A}x_k + \boldsymbol{C}z^{(i)}_k + \boldsymbol{\zeta})
\end{align*}
and estimate the value of executing the $i$th action by computing $W^\top x^{(i)}_{k+1}$. Then we choose to execute the action $a^*$ with the greatest estimated value - which determines our new policy $\pi_1$ - which yields a significant improvement over the initial policy $\pi_0$, as illustrated in Figure \ref{fig:trained_bee_main}. Under the initial policy $\pi_0$ the bee collected an average of approximately 1.05 nectar per unit time, in comparison to 1.52 nectar under the improved policy $\pi_1$. This is much closer to the optimal value of approximately 1.60, which we obtain in the next section. The algorithm which first approximates the value function, and then updates the policy is described in Algorithm \ref{1_step_algorithm}.

\begin{figure}[t!]
    \centering
    \begin{subfigure}[b]{1\textwidth}
        \includegraphics[width=\textwidth]{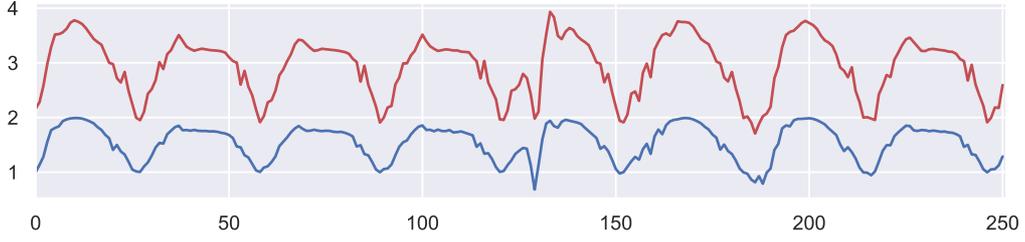}
        \caption{The nectar collected (blue) and the approximate value function (red) is plotted for the first 250 time steps (y-axis) under the improved policy $\pi_1$.}
        \label{fig:trained_bee_plot}
    \end{subfigure}

    \begin{subfigure}[b]{1\textwidth}
        \includegraphics[width=1\textwidth]{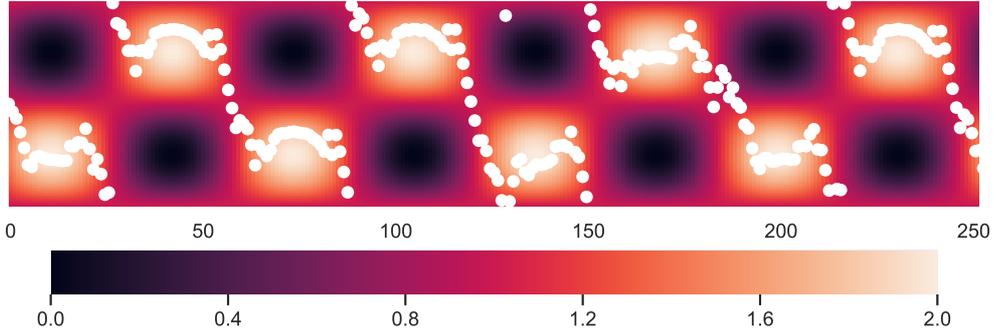}
        \caption{The nectar function $n(y,t)$~\eqref{eqn:nectar} at every point represented as a heat map in the $(t,y)$ plane, with the position of the bee at time $t$ under the improved policy is indicated by the overlaid white circles.}
        \label{fig:trained_bee}
    \end{subfigure}
\caption{Dynamics of Bee World where the bee executes the improved policy $\pi_1$ for the first $250$ time steps.}
\label{fig:trained_bee_main}
\end{figure}

\begin{algorithm}
\caption{One Step Offline Learning Algorithm (Bee World)}
\label{1_step_algorithm}
\begin{algorithmic}[1]
\State Choose initial reservoir state $x_0$
\State Randomly generate $\boldsymbol{A},\boldsymbol{C},\boldsymbol{\zeta}$
\State \textbf{for} each $k$ from $0$ to $\ell-1$
\Indent 
    \State Compute $x_{k+1} = \sigma(\boldsymbol{A}x_k + \boldsymbol{C}(r_k,a_k) + \boldsymbol{\zeta})$
\EndIndent
\State Find $W$ that minimises $\sum^{\ell-1}_{k=0} \lVert W^{\top}(x_k - \gamma x_{k+1}) - r_k \rVert^2 + \lambda \lVert W \rVert^2$ 
\State \textbf{for} each $k$ from $\ell$ to $L-1$
\Indent 
    \State Compute $a^* = \max_a \{ W^{\top}\sigma(\boldsymbol{A}x_k + \boldsymbol{C}(r_k,a) + \boldsymbol{\zeta}) \}$
    \State Compute $x_{k+1} = \sigma(\boldsymbol{A}x_k + \boldsymbol{C}(r_k,a^*) + \boldsymbol{\zeta})$
\EndIndent
\end{algorithmic}
\end{algorithm}

\subsection{An Analytic Solution for Bee World}

In this section, we will analyse Bee World so that we can compare the ESN solution to results that we can prove. To make our own lives easier, we consider a smooth version of Bee World, rather than the discrete time version solved by the ESN, so that we can formulate Bee World as a control problem that admits a solution via the Euler-Lagrange equation.
We have the control system
\begin{align*}
    \dot{\tau} &= 1 \\
    \dot{y} &= u(y,\tau)
\end{align*}
where $u$ is the controller dependant on $y$ and $\tau$. Then we have a cost function
\begin{align*}
    \mathcal{C}(x,\tau,u) = f(u) - n(y,\tau)
\end{align*}
where $f(x)$ is the penalty term for using the control $u$ and $n(y,\tau)$ is the nectar function. In the above formulation of Bee World 
\begin{align*}
    f(u) = 
    \begin{cases}
        0 &\text{ if } -c \leq u \leq c \\
        \infty &\text{ otherwise } \\
    \end{cases}
\end{align*}
where $c = 0.1$. Then the objective is to find
\begin{align*}
    u^* = \argmin_u \bigg\{ \int_{0}^{\infty} \gamma^t \mathcal{C}(y,\tau,u) \ dt \bigg\}.
\end{align*}
We can see that $f$ is not a well defined function so we will introduce the family of functions
\begin{align*}
    f_{\epsilon}(u) = -\epsilon \log(\cos(\pi u/(2c)))
\end{align*}
where $\epsilon > 0$, and notice that $f_\epsilon$ approaches $f$ pointwise as $\epsilon \to 0$. Next, we recall that the stationary points (including the minimum) of the integral functional
\begin{align*}
    \mathcal{I}[y] = \int^{\infty}_0 \mathcal{F}(t,y,\dot{y}) \ dt
\end{align*}
all satisfy the Euler-Lagrange equation
\begin{align*}
    \frac{d}{dt}\frac{\partial \mathcal{F}}{\partial \dot{y}} - \frac{\partial \mathcal{F}}{\partial y} = 0.
\end{align*}
So, we let
\begin{align*}
    \mathcal{F}(t,y,\dot{y}) &= \gamma^t \mathcal{C}(t,y,\dot{y}) \\
                             &= \gamma^t (-\epsilon \log(\cos(\pi\dot{y}/(2c))) - \cos(\omega t)\sin(2 \pi y) - 1 )
\end{align*}
then
\begin{align*}
    0 &= \frac{d}{dt}\frac{\partial \mathcal{F}}{\partial \dot{y}} - \frac{\partial \mathcal{F}}{\partial y} \nonumber \\
    &= \frac{d}{dt}\bigg( \gamma^t \frac{d}{d\dot{y}}(-\epsilon \log(\cos(\pi\dot{y}/(2c)))) \bigg) + 2\pi \gamma^t \cos(\omega t) \cos(2 \pi y) \nonumber \\
    &= \frac{\pi\epsilon}{2c} \frac{d}{dt}\bigg( \gamma^t \tan(\pi\dot{y}/(2c)) \bigg) + 2\pi \gamma^t \cos(\omega t) \cos(2 \pi y) \nonumber \\
    &= \frac{\pi\epsilon}{(2c)} \bigg( \log(\gamma) \gamma^t \tan(\pi\dot{y}/(2c)) + \gamma^t \frac{\pi\ddot{y}}{2c}\sec^2(\pi\dot{y}/(2c)) \bigg) + 2\pi \gamma^t \cos(\omega t) \cos(2 \pi y) \nonumber \\
    &= \frac{\pi\epsilon}{2c} \bigg( \log(\gamma) \tan(\pi\dot{y}/(2c)) + \frac{\pi\ddot{y}}{2c}\sec^2(\pi\dot{y}/(2c)) \bigg) + 2\pi \cos(\omega t) \cos(2 \pi y) \nonumber,
\end{align*}
which we can reformulate as a dynamical system
\begin{align}
    \dot{v} &= -\frac{2c \cos^2(\pi v/(2c))}{\pi}\bigg(\frac{4 c \cos(\omega \tau) \cos(2 \pi y)}{\epsilon} + \log(\gamma) \tan(\pi v/(2c)) \bigg) \nn \\
    \dot{y} &= v \nn \\
    \dot{\tau} &= 1 \label{optimal_bee_dynamics}
\end{align}
whose solutions are stationary points of the integral functional. For small $\epsilon$, we approach the Bee World problem. We took $\epsilon = 10^{-5}$, $\gamma = 1/2$, initial position $y = 0$, and initial velocity $v = 0$ then simulated a trajectory of the ODE using \codeword{scipy.integrate.odeint}. We plotted this in Figure \ref{fig:optimal_bee}. The average nectar collected by under this policy was approximately 1.60.

\begin{figure}[t]
    \centering
        \includegraphics[width=\textwidth]{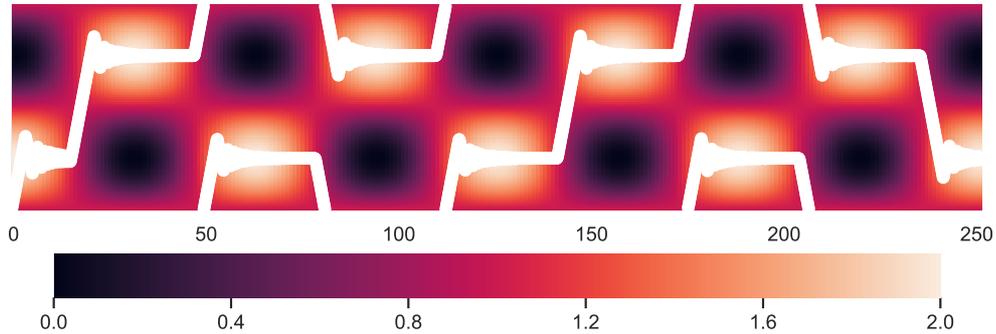}
        \caption{A numerical solution to the ODE \eqref{optimal_bee_dynamics} with $\epsilon = 10^{-5}$ (white line) superposed on the heat map of the nectar function $n(y,t)$ given in~\eqref{eqn:nectar}. Dark colours indicate regions of low nectar, light regions indicate high values of the nectar function. We observe that the solution trajectory spends much more time near local maxima of the nectar function but has complicated oscillatory fluctuations during transitions between local maxima. The oscillations are likely due to approaching a sort of singularity as $\epsilon \to 0$.}
        \label{fig:optimal_bee}
\end{figure}

\section{Application to Stochastic Control}
\label{a_stochastic_optimal_control_problem}

ESNs have shown remarkable promise in solving problems in mathematical finance - including by \cite{LIN20097313}, \cite{doi:10.1260/1748-3018.7.1.87}, and \cite{Dan2014} who used an ESN to predict the future values of stock prices.
\cite{6319485} used an ESN to learn the solution to a credit rating problem and
\cite{6924052} used an ESN to forecast exchange rates, comparing the results to forecasts made with an ARMA model. 
In this section we will introduce a stochastic optimal control problem arising in the market making problem. We will solve this problem analytically, and compare this to the solution obtained by a reinforcement learning agent supported by an ESN.

\subsection{A Market Making Problem}

We consider a stochastic control problem inspired by the motivations of a market maker acting in a general financial market.   In practice the specific role of a market maker depends on the particular market, but we consider a market maker who provides liquidity to other market participants by quoting prices at which they are willing to sell (ask) and buy (bid) an asset. By setting the ask price higher than the bid price in general they can profit from the difference when they receive both a buy and sell order at these prices. However, the market maker faces risk, since if they buy a quantity of the asset the market price might move against them before they are able to find a seller. 

The market making problem is a complex one, and has been studied extensively since the publication of the paper by \cite{doi:10.1080/14697680701381228}.  The paper of \cite{doi:10.1080/1350486X.2017.1342552} gives a good overview of much of this work. We consider a stylised version of this problem that focuses on inventory management without considering explicit optimal quoting strategies. We consider that a market maker acting relatively passively around the market price in ordinary conditions would expect to observe a random demand for buy and sell orders. If as a result of random fluctuations they find their inventory has drifted away from zero, they would set prices more competitively on either the ask or bid side to encourage trades to balance their position. Very broadly the conclusions of work on the market making problem are that there is a price to be paid to exert control over the inventory process and bring inventories closer to zero.   

Motivated by this insight, we consider the market maker's inventory to be a stochastic process $(\boldsymbol{Y}_t)_{t \geq 0}$ with dynamics
\begin{align*}
    d\boldsymbol{Y}_t=\boldsymbol{\pi}_t dt + \sigma d\boldsymbol{W}_t
\end{align*}
where $(\boldsymbol{W}_t)_{t \geq 0}$ is a standard Brownian motion.

The parameter $\sigma$ measures the volatility of the incoming order flow, and $(\boldsymbol{\pi}_t)_{t \geq 0} $ is the control process by which the market maker adds drift into their order flow by moving their bid and ask quotes. Naturally, there is a cost involved in applying the control, and a further cost to holding inventory away from zero. We introduce parameters $\alpha$ and $\beta$ to quantify these effects and model the market maker's profit as a stochastic process solving
$$d\boldsymbol{Z}_t=(r-\alpha \boldsymbol{\pi}_t^2-\beta \boldsymbol{Y}_t^2) dt$$
where $r$ is the rate of profit the market maker would achieve from the bid-ask spread if they did not have concerns about the asset price movements. We consider the case where the market maker seeks to maximise their long run discounted profit

$$v(y)=\max_{\pi} \mathbb{E}^y\Big[\int_0^\infty e^{-{\delta t}} d\boldsymbol{Z}_t \Big],$$
where $\mathbb{E}^y$ is the expectation with the process started at $Y_0=y$. 
We can show that the market maker's value function and optimal control are
\begin{align}
    v(y)= - \alpha hy^2+\frac{r- \alpha h \sigma^2}{\delta}, \qquad \pi^* (y) = - hy , \label{v_and_phi}
\end{align}
where $$h:= \frac{-\alpha \delta + \sqrt{\alpha^2 \delta^2+4 \beta}}{2 \alpha}$$

Further, the inventory process $\boldsymbol{Y}_{t \geq 0}$, when controlled by the optimal control $\pi^*(y)=-hy$ is given by the Ornstein-Uhlenbeck process
$$d\boldsymbol{Y}_t=-h\boldsymbol{Y}_t dt+ \sigma d\boldsymbol{W}_t$$
whose stationary distribution is a Gaussian $\mathcal{N}\left(0, \frac{\sigma^2}{2h}\right)$.

We observe that this is an infinite horizon, Linear-Quadratic regulator (LQR) type problem, a class of problems which have a long history in the control literature, and more recently have been systematically studied in the reinforcement learning literature. Recent work on online learning for the LQR problem (e.g. \cite{fazel_global_2018,learning_the_LQR,dean_sample_2020}) has considered a range of variants of the LQR problem, including cases with uncertainty on the both the dynamics and the reward, and where the state variable may only be partially observed. However most of these approaches work in the setting of model-based learning approaches: that is, they attempt to learn a ``model'' of the world, and therefore exploit the fact that the LQR structure is known and can be learned from the data; in comparison, \cite{fazel_global_2018} still rely on the LQR structure, but do not directly try to learn the ``model'' of the world. The paper \cite{tu_gap_2019} analyses the difference between model-based and model-free approaches to the LQR problem, showing that one should expect an exponential separation between model-based and model-free approaches.  In this context, our approach, which does not assume the LQR structure, can also be compared to model-free approaches, such as the classical work of \cite{bradtke_adaptive_1994}, which takes a $Q$-learning approach.
% so our attempts to learn the solution with an ESN is highly related to the work by \cite{learning_the_LQR}. 

\subsection{Discretised problem}

To turn this into a problem into one that can be used to train an Echo State Network we reformulate it in discrete time; we consider a process $\boldsymbol{Y}_0, \boldsymbol{Y}_1,\boldsymbol{Y}_2,\ldots $ such that 
$$\boldsymbol{Y}_{k+1}-\boldsymbol{Y}_k= \epsilon \boldsymbol{\pi}_k  + \sigma\sqrt{\epsilon} \mathcal{N}_k$$
where $(\mathcal{N}_k)_{k \in \mathbb{N}}$ are a sequence of i.i.d. random variables $\mathcal{N}_k \sim \mathcal{N}(0,1)$ for each $k \in \mathbb{N}$, and $\epsilon > 0$ is the time increment.  The control is now a sequence $\pi=(\boldsymbol{\pi}_k)_{k \in \mathbb{N}}$.  The profit function satisfies $\boldsymbol{Z}_0=0$ and 
$$d\boldsymbol{Z}_k := \boldsymbol{Z}_{k+1}-\boldsymbol{Z}_k=\epsilon(r-\alpha \boldsymbol{\pi}_k^2-\beta \boldsymbol{Y}_k^2).$$
and the market maker seeks to maximise the value function

$$v(y)=\max_{\pi} \mathbb{E}^y\Big[\sum_{k=0}^\infty e^{-{\delta \epsilon k}} d\boldsymbol{Z}_k \Big],$$
over choices of the control $\pi$ where $\mathbb{E}^y$ is the expectation with the process started at $\boldsymbol{Y}_0=y$.

It can be shown that in the limit as $\epsilon \to 0$, the optimal control and value function for this problem converge precisely to the optimal control and value function in the continuous case.

We state here the results in the case $\epsilon=1$, the value we will use for the application of the Echo State Network below. Writing $\gamma=e^{-\delta}$, we find in this case that the value function and optimal control are given by
$$v(y)=-\alpha p y^2+\frac{r-\gamma \alpha p\sigma^2}{1-\gamma}, \qquad \pi^*= - p y$$
where $$p:= \frac{(\alpha (\gamma-1)+\gamma \beta) + \sqrt{(\alpha (\gamma-1)+\gamma \beta)^2 + 4 \alpha \beta \gamma}}{2\gamma \alpha}$$.

%where
%$$ k = \frac{\gamma a }{\gamma a-\alpha} $$
%$$a=\frac{-(\alpha (\gamma-1)+\gamma \beta) \pm \sqrt{(\alpha (\gamma-1)+\gamma \beta)^2 + 4 \alpha \beta \gamma}}{2\gamma}$$
%and 
%$$c=\frac{r+\gamma a\sigma^2}{1-\gamma}$$

The process $\boldsymbol{Y}$ controlled by $\pi^*$ is Markovian, and has transition operator
\begin{align*}
    (\mathcal{T}s)(y) &= \int_{-\infty}^{\infty} \mathbb{P}(\boldsymbol{Y}_{k+1} = y \ | \ \boldsymbol{Y}_k = x) s(x) \, dx \\
     &= \frac{1}{\sqrt{2\pi}\sigma} \int_{-\infty}^{\infty} e^{-\frac{(y-(1-p x))^2}{2\sigma^2}}s(x) \, dx.
\end{align*}

It is straightforward to verify that the Gaussian probability density function
\begin{align}
    s^*(y) = \frac{\sqrt{p(2-p)}}{\sqrt{2\pi}\sigma}\mathrm{e}^{-\frac{y^2p(2-p)}{2\sigma^2}}, \label{invariant_measure}
\end{align}
is a fixed point of $\mathcal{T}$ and hence that the controlled process has stationary distribution $\mathcal{N}\left(0,\frac{\sigma^2}{p(2-p)}\right)$.

\subsection{Solving the Market Making Problem with an ESN}
\label{solving_the_mm_problem_with_an_ESN}

%The game of Stock World evolves in discrete time steps $t = 0, 1, 2...$, and features a \emph{market maker} who owns an inventory of assets whose total value takes values in $\mathbb{R}$. Without any intervention from the market maker, the value of their assets follows a discrete time Brownian motion - representing the moment to moment fluctuations of the economy. There are costs associated to holding net positive or net negative inventory, so the market maker is incentivised to keep the inventory close to 0 by buying and selling assets at the right time for the right price. This buying and selling also incurs costs, so the market maker must balance the costs of straying too far from 0 with the costs of executing transactions to minimise the total cost of operation over time.

In this section, we seek to solve the the market making problem with a reinforcement learning algorithm supported by an ESN. In this set up, we assume the market maker has no knowledge of the cost function, and no knowledge of the effect of executing an action. The agent must execute a variety of actions in a variety of states to learn about the environment and the effect of its actions. Then, the market maker makes reasonable changes to its policy to arrive at a policy that reduces the long term costs of operation. The policy obtained by the reinforcement learning approach is compared to the optimal policy derived with full knowledge of the system.

\subsubsection{Approximating the value functional}
\label{approx_the_val_func}

For the purpose of running the simulation, we let the cost of operating the control $\alpha = 1$, the cost of straying from the origin $\beta = 1$, the timestep $\epsilon = 1$, and the volatility parameter $\sigma = 1$. We take the baseline profit parameter $r = 0$.
The inventory held, and action taken, by the market maker at time $k$ will be denoted $y_k$ and $a_k$ respectively. A sequence of (inventory, action) pairs will be denoted $z \in (\mathbb{R}^2)^{\mathbb{Z}}$ with $z_k = (y_k,a_k)$.
The value functional for the market maker problem is defined
\begin{align*}
    V(z) = \mathbb{E}_{\mu}\bigg[ \sum_{k=0}^{\infty} \gamma^k \mathcal{R}T^k(\boldsymbol{Z}) \ \bigg| \ \boldsymbol{Z}_j = z_j \ \forall j \leq 0 \bigg]
\end{align*}
where $\mathcal{R}: (\mathbb{R}^2)^{\mathbb{Z}} \to \mathbb{R}$ is the reward functional
\begin{align*}
    \mathcal{R}(z) = -(\alpha a_{-1}^2 + \beta y_0^2),
\end{align*}
$T$ is the shift operator, and $\gamma \in [0,1)$ is the discount factor representing the relative importance of near and long term costs.  We can see after a simple rearrangement that 
\begin{align*}
    V(z) = \mathcal{R}(z) + \gamma \mathbb{E}_{\mu}[VT_{\boldsymbol{Z}}(z)]
\end{align*}
so $V$ is the unique fixed point of the contraction mapping $\Phi$ defined by
\begin{align*}
    \Phi (H)(z) = \mathcal{R}(z) + \gamma \mathbb{E}_{\mu}[HT_{\boldsymbol{Z}}(z)]
\end{align*}
as discussed in Section \ref{training_ESNs_with_least_squares}. Thus, by Theorem \ref{offline_q_learning}, we can approximate the value function $V$ using an ESN trained by regularised least squares \emph{if} the (inventory, action) pairs $(y_k,a_k)$ are the realisation of a stationary ergodic process. 
Consequently, we sought an initial policy $\pi_0$ such that the process $\boldsymbol{Z}$ comprising the inventory-action pairs under policy $\pi_0$ is stationary and ergodic. In particular, we chose
\begin{align}
    \pi_0(y) \sim \mathcal{N}(0,\sigma_i^2) - \eta y \label{initial_policy_market_maker}
\end{align}
with $\eta = 0.05$ a constant representing the rate of exponential drift toward $0$ and $\sigma_i^2 = 1$.
We ran this policy for 10000 time steps, and recorded the pairs $z_k$ along with the rewards $r_k$.
Next, we set up an ESN of dimension $n = 300$, with reservoir matrix, input matrix, and bias $\boldsymbol{A},\boldsymbol{C},\boldsymbol{\zeta}$ populated with i.i.d uniform random variables $U(-0.05,0.05)$. $\boldsymbol{A}$ was then multiplied by a scaling factor such that the 2-norm of $\boldsymbol{A}$ satisfies $\lVert \boldsymbol{A} \rVert_2 = 1$. As in the previous example we chose $\sigma$ to be the ReLU activation function.
We then computed reservoir states 
\begin{align*}
    x_{k+1} = \sigma(\boldsymbol{A} x_k + \boldsymbol{C} z_k + \boldsymbol{\zeta})
\end{align*}
starting with an initial reservoir state $x_0 = 0$. An arbitrary reservoir state $x$ then encodes the left infinite sequence of (inventory,action) pairs $z$.
We seek an expression for the value of the reservoir state $x$ by solving the least squares problem
\begin{align*}
    W = (\Xi^\top \Xi + \lambda I)^{-1} \Xi^\top U  
\end{align*}
(using the singular value decomposition)
where $\Xi$ is the matrix with $k$th column is
\begin{align*}
    \Xi_k := x_{k} - \gamma x_{k+1}
\end{align*}
and $U$ is the vector of observations where the $k$th entry is the reward $r_k$, and $\lambda$ is the regularisation parameter which we set to \codeword{1e-6}. We also chose $\gamma = e^{-1}$. In practice, the discount factor is usually much larger.
With this, we obtain an expression for value of the reservoir state $x$ given by $W^{\top} x$. The results of this policy are shown in Figures  \ref{initial_value} and \ref{fig:untrained_market_maker_main}. The procedure which estimates the value function and improves upon the policy is described in Algorithm \ref{1_step_algorithm_mm}.

\begin{figure}
    \centering
        \includegraphics[width=0.75\textwidth]{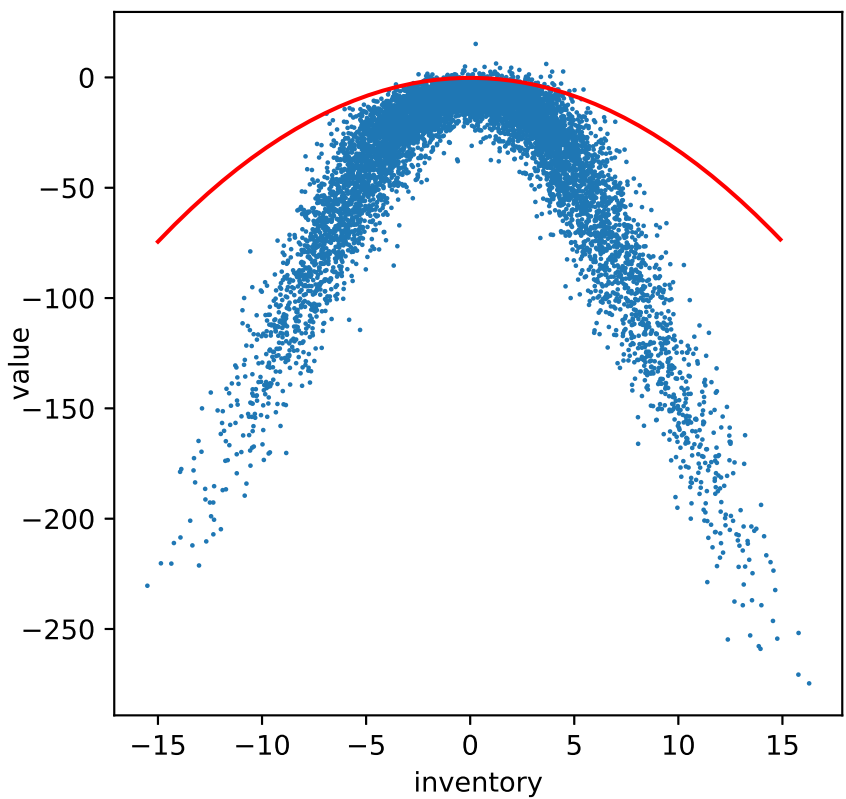}
        \caption{Under the initial policy, the value $V(\boldsymbol{Y})$ ($y$-axis) learned by the ESN at the inventory $\boldsymbol{Y}$ ($x$-axis) at each of the 10000 timesteps is shown. The parabolic shape is consistent with the analytically derived optimal value function~(\ref{invariant_measure}) shown in red. We note that the value function under the initial policy $\pi_0$ is not expected to match the value function under the optimal policy $\pi^*$.}
        \label{initial_value}
\end{figure}

\begin{figure}
    \centering
        \begin{subfigure}[b]{1.0\textwidth}
        \includegraphics[width=1.0\textwidth]{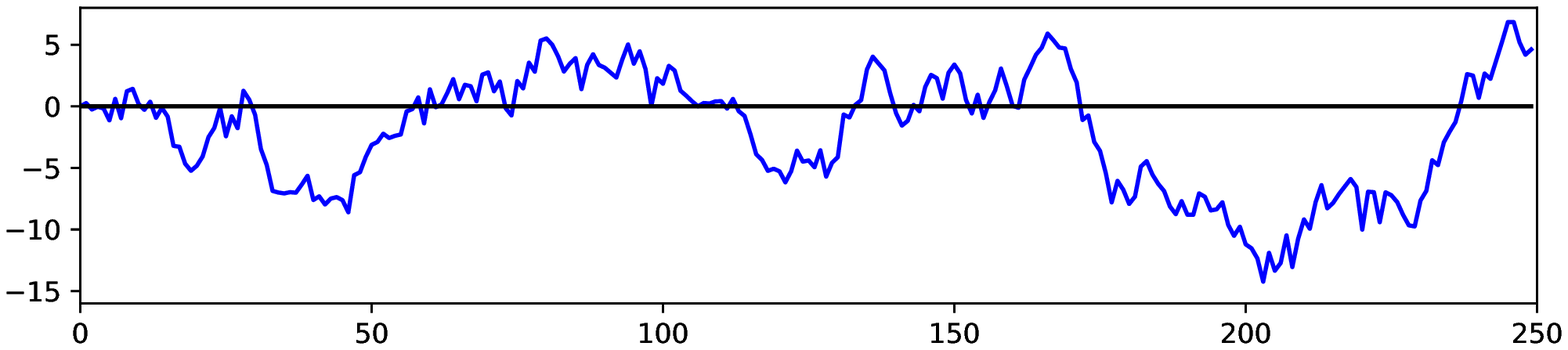}
        %\caption{Market maker executing the initial policy}
        \caption{}
        \label{fig:untrained_market_maker}
    \end{subfigure}
    \begin{subfigure}[b]{1.0\textwidth}
        \includegraphics[width=1.0\textwidth]{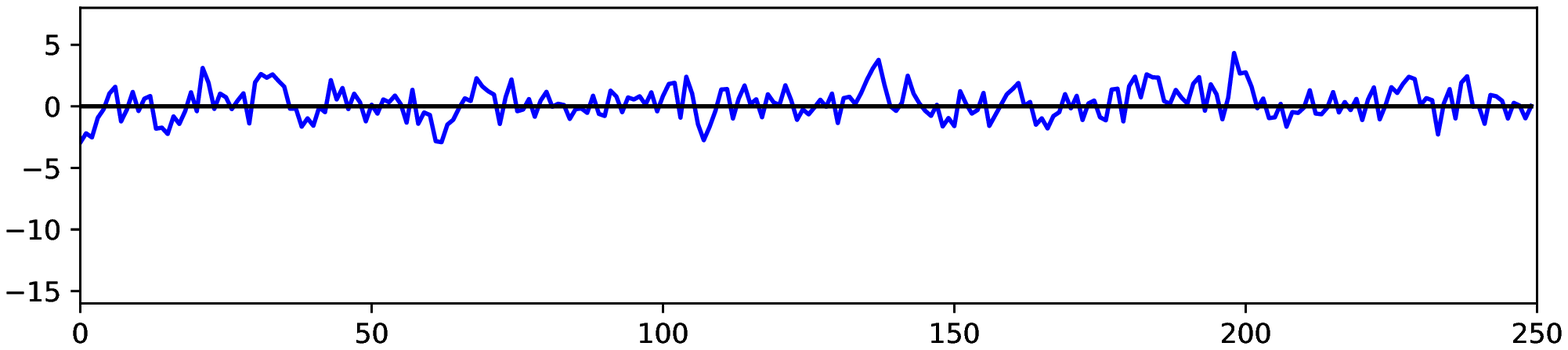}
        %\caption{Market maker executing the improved policy}
        \caption{}
        \label{fig:trained_market_maker}
    \end{subfigure}
\caption{
Dynamics of the market maker over time executing (a) the initial policy $\pi_0$ 
%(\ref{fig:untrained_market_maker})
and (b) the improved policy $\pi_1$.  %(\ref{fig:trained_marketmaker}).
For each plot, the inventory ($y$-axis)
is shown evolving with time ($x$-axis).}
\label{fig:untrained_market_maker_main}
\end{figure}

\subsubsection{Updating the policy}
\label{update_the_policy}

We sought to create a new and improved policy based on the observations of under the initial policy using a na\"ive approach. At each time step, we consider 100 trial actions $a^{(1)}, a^{(2)}, \ldots  , a^{(100)}$ drawn from the standard normal distribution $\mathcal{N}(0,1)$ and compute 
\begin{align*}
    x^{(i)}_{k+1} = \sigma(\boldsymbol{A} x_k + \boldsymbol{C} z^{(i)}_k + \boldsymbol{\zeta})
\end{align*}
where $z^{(i)}_k$ is the (inventory, action) pair $(y_k, a^{(i)})$, and $a^{(i)}$ is trial action. For each $i$, we compute $W^{\top} x^{(i)}_{k+1}$ to obtain the predicted value of executing action $a^{(i)}$. We then choose to execute the action $a^*$ with the greatest predicted value, and update the reservoir state using this (inventory, action) pair $(y_k,a^*)$. This defines our new policy. We ran this new policy for 10,000 time steps and illustrated the results in Figures \ref{fig:scatter}, and \ref{fig:invariant_measure}.

\begin{figure}[t]
    \centering
    \begin{subfigure}{0.47\textwidth}
        \includegraphics[width=\textwidth]{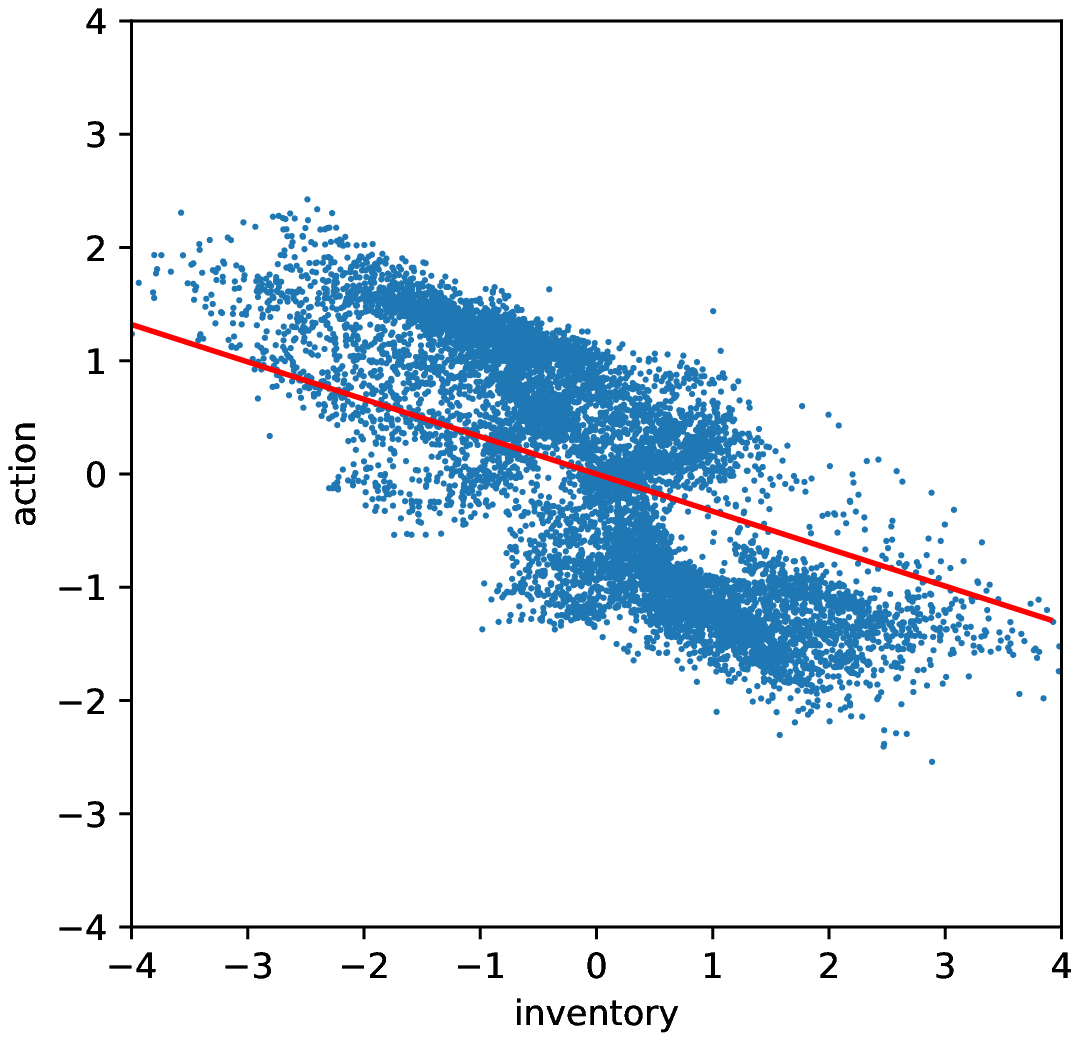}
        \caption{}
        \label{fig:scatter}
        \end{subfigure}
        ~
        \begin{subfigure}{0.47\textwidth}
            \includegraphics[width=\textwidth]{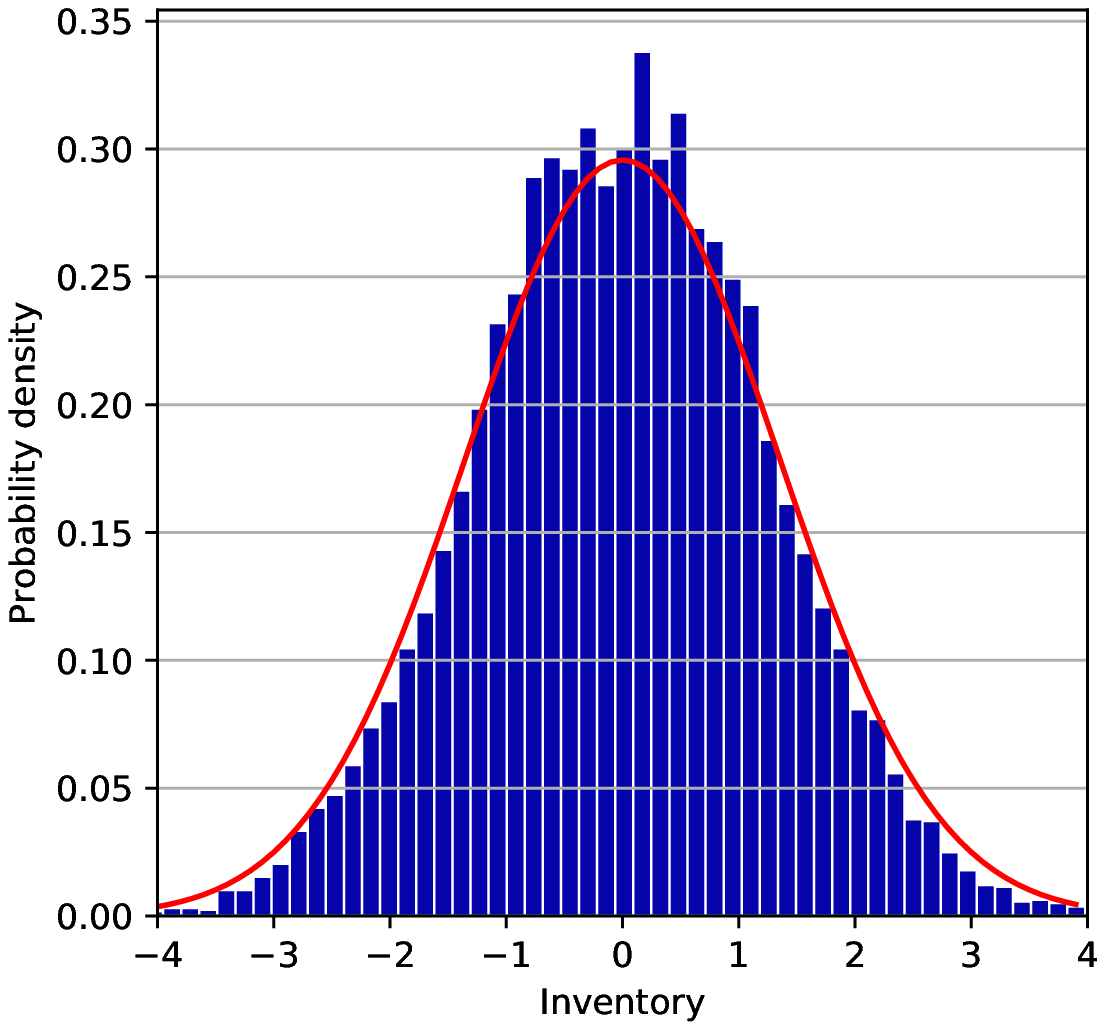}
            \caption{}
        \label{fig:invariant_measure}
        \end{subfigure}
        \caption{(a) Illustrates the (inventory, action) pairs $(y_k, a_k)$ under the improved policy $\pi_1$ are represented as points on the scatter plot. The inventory is on the $x$-axis, and action is on the $y$-axis. The red line represents the analytically derived optimal control (equation   \eqref{v_and_phi}). (b) Illustrates the invariant measure of the inventory process under the improved policy $\pi_1$ is approximated with a histogram. The histogram is compared to the analytically derived invariant measure of the optimal control process $\mathcal{N}(0,1.82)$ (equation \eqref{invariant_measure}).}
\end{figure}

\begin{algorithm}
\caption{One Step Offline Learning Algorithm (Market Making)}
\label{1_step_algorithm_mm}
\begin{algorithmic}[1]
\State Choose initial reservoir state $x_0$
\State Randomly generate $\boldsymbol{A},\boldsymbol{C},\boldsymbol{\zeta}$
\State \textbf{for} each $k$ from $0$ to $\ell-1$
\Indent 
    \State Compute $x_{k+1} = \sigma(\boldsymbol{A}x_k + \boldsymbol{C}(y_k,a_k) + \boldsymbol{\zeta})$
\EndIndent
\State Find $W$ that minimises $\sum^{\ell-1}_{k=0} \lVert W^{\top}(x_k - \gamma x_{k+1}) - r_k \rVert^2 + \lambda \lVert W \rVert^2$ 
\State \textbf{for} each $k$ from $\ell$ to $L-1$
\Indent 
    \State Compute $a^* = \max_a \{ W^{\top}\sigma(\boldsymbol{A}x_k + \boldsymbol{C}(y_k,a) + \boldsymbol{\zeta}) \}$
    \State Compute $x_{k+1} = \sigma(\boldsymbol{A}x_k + \boldsymbol{C}(y_k,a^*) + \boldsymbol{\zeta})$
\EndIndent
\end{algorithmic}
\end{algorithm}

\subsection{Comparison between the analytic and learned solutions}

The one step reinforcement learning algorithm did not perfectly replicate the analytically derived optimal control, but has moved in a promising direction. We can see in Figure \ref{fig:scatter} that the inventory process under the improved policy produces (inventory, action) pairs that have some scatter relative to the optimal policy indicated by the red straight line. This suggests that the market maker trained by reinforcement learning is behaving well in some average sense, despite performing many sub-optimal actions. It also appears that the the reinforcement learning algorithm uses the control more aggressively than is optimal. This sub-optimal control results in greater costs than the optimal control. In particular the average cost incurred under the improved policy $\pi_1$ is 2.65, while the average cost under the the analytically derived optimal policy is $\sigma / \sqrt{p(2-p}) = 1.35$. 

Despite these sub-optimal moves, it seems that the inventory process learned by the market maker has an invariant measure that closely matches the optimal invariant measure. It is reassuring to see that an invariant measure appears, at least numerically, to exist, because the controlled process is assumed to be stationary and ergodic (and therefore admits an invariant measure) in Theorem \ref{offline_q_learning}.

It is also worth noting that the inventory process, controlled either by the ESN or the optimal control, has support on $\mathbb{R}$, which is not a compact space. Therefore, the conditions of Theorem \ref{offline_q_learning} don't technically hold. However, the numerical results here suggest that the ESN has learned the value functional adequately well, suggesting that Theorem \ref{offline_q_learning} may hold under relaxed conditions. Of course, realisations of the stochastic processes always explore only bounded subsets of $\mathbb{R}$.

\section{Conclusions and future work}
\label{Conclusions}

In this paper we have presented three novel mathematical results concerning Echo State Networks trained on data drawn from a stationary ergodic process.
The first applies to offline supervised learning. The theorem states that, given a target function, enough training data and a large enough ESN, the least squares training procedure will yield an arbitrarily good approximation to the target function. The second result applies to an agent performing a stochastic policy $\pi$. After the agent has collected enough training data, and given a sufficiently large ESN, the least squares training procedure will yield an arbitrarily good approximation to the value function associated to the policy $\pi$.
The third result is relevant to online reinforcement learning. Though the result is quite preliminary, the lemma is introduced with the intention of developing online algorithms (inspired by Q-learning) to learn the optimal policy for non-Markovian problems. 

We demonstrated the second result (which generalises the first) on a deterministic control problem (Bee World) and a stochastic control problem (the market making problem). We chose these `toy model' problems to understand the performance of the algorithm completely in cases that are solvable analytically, although these optimal solutions themselves are not entirely trivial. 
The reinforcement learning algorithm we use to improve the policy in both Bee World and the market making problem is extremely simple. It is essentially one iteration of an $\epsilon$-greedy policy \citep{the_bible}, with $\epsilon$ set to 0. Despite the simplicity of the algorithm, the single iteration considerably improved the policy, resulting in a reasonable approximation to the optimal policy. 

It therefore seems a natural direction of future work to develop more sophisticated learning algorithms. Notably the linear upper confidence bound (linUCB) algorithm \citep{the_bible} has a linear structure that fits cleanly into the the linear training framework of the ESN. As this work develops, it will become essential to have a rigorous framework describing the relationship between filters, functionals, random processes and reinforcement learning.  The theory presented in this paper tentatively connects these objects using ideas from Markov Decision Processes, but the theory is far from complete.

\section*{Acknowledgements}
Allen Hart and Kevin Olding are funded through the EPSRC Centre for Doctoral Training in Statistical Applied Mathematics at Bath (SAMBa), grant number EP/L015684/1.

We thank Jeremy Worsfold for insights about reinforcement learning and the linUCB algorithm, and for refactoring the Bee World code. We also thank Adam White for useful discussions about reinforcement learning settings.

\bibliographystyle{elsarticle-num} 
\bibliography{references}

\begin{thebibliography}{10}
\expandafter\ifx\csname url\endcsname\relax
  \def\url#1{\texttt{#1}}\fi
\expandafter\ifx\csname urlprefix\endcsname\relax\def\urlprefix{URL }\fi
\expandafter\ifx\csname href\endcsname\relax
  \def\href#1#2{#2} \def\path#1{#1}\fi

\bibitem{Jaeger2001}
H.~Jaeger, The “echo state” approach to analysing and training recurrent
  neural networks, GMD-Report 148, German National Research Institute for
  Computer Science (01 2001).

\bibitem{doi:10.1162/089976602760407955}
W.~Maass, T.~Natschläger, H.~Markram, Real-time computing without stable
  states: A new framework for neural computation based on perturbations, Neural
  Computation 14~(11) (2002) 2531--2560.
\newblock \href {https://doi.org/10.1162/089976602760407955}
  {\path{doi:10.1162/089976602760407955}}.

\bibitem{General_Value_Function_Networks}
M.~Schlegel, A.~Jacobsen, M.~Zaheer, A.~Patterson, A.~White, M.~White, General
  value function networks, arXiv:1807.06763 (2018).

\bibitem{Gonon2020}
L.~Gonon, L.~Grigoryeva, J.-P. Ortega, Approximation bounds for random neural
  networks and reservoir systems, arXiv:2002.05933 (2020).

\bibitem{LUKOSEVICIUS2009127}
M.~Lukoševičius, H.~Jaeger, Reservoir computing approaches to recurrent
  neural network training, Computer Science Review 3~(3) (2009) 127 -- 149.
\newblock \href {https://doi.org/https://doi.org/10.1016/j.cosrev.2009.03.005}
  {\path{doi:https://doi.org/10.1016/j.cosrev.2009.03.005}}.

\bibitem{LukoseviciusMantas2012RCT}
M.~Luko{\v s}evi{\v c}ius, H.~Jaeger, B.~Schrauwen, Reservoir computing trends,
  K{\"u}nstliche Intelligenz. 26~(4) (2012) 365--371.

\bibitem{RodanA2011MCES}
A.~Rodan, P.~Tino, Minimum complexity echo state network, IEEE transactions on
  neural networks 22~(1) (2011) 131--144.

\bibitem{NIPS2010_4056}
F.~Triefenbach, A.~Jalalvand, B.~Schrauwen, J.-P. Martens, Phoneme recognition
  with large hierarchical reservoirs, in: J.~D. Lafferty, C.~K.~I. Williams,
  J.~Shawe-Taylor, R.~S. Zemel, A.~Culotta (Eds.), Advances in Neural
  Information Processing Systems 23, Curran Associates, Inc., 2010, pp.
  2307--2315.

\bibitem{Jaeger78}
H.~Jaeger, H.~Haas, Harnessing nonlinearity: Predicting chaotic systems and
  saving energy in wireless communication, Science 304~(5667) (2004) 78--80.
\newblock \href
  {http://arxiv.org/abs/https://science.sciencemag.org/content/304/5667/78.full.pdf}
  {\path{arXiv:https://science.sciencemag.org/content/304/5667/78.full.pdf}},
  \href {https://doi.org/10.1126/science.1091277}
  {\path{doi:10.1126/science.1091277}}.

\bibitem{10.1007/11840817_86}
I.~Szita, V.~Gyenes, A.~L{\H{o}}rincz, Reinforcement learning with echo state
  networks, in: S.~D. Kollias, A.~Stafylopatis, W.~Duch, E.~Oja (Eds.),
  Artificial Neural Networks -- ICANN 2006, Springer Berlin Heidelberg, Berlin,
  Heidelberg, 2006, pp. 830--839.

\bibitem{Pietz_2021}
S.~Peitz, K.~Bieker, On the universal transformation of data-driven models to
  control systems, arXiv:2102.04722 (2021).

\bibitem{Random_Weights}
A.~Saxe, P.~W. Koh, Z.~Chen, M.~Bhand, B.~Suresh, A.~Ng, On random weights and
  unsupervised feature learning, In Proceedings of the 28th International
  Conference on Machine Learning (2011).

\bibitem{GRIGORYEVA2018495}
L.~Grigoryeva, J.-P. Ortega, Echo state networks are universal, Neural Networks
  108 (2018) 495 -- 508.
\newblock \href {https://doi.org/https://doi.org/10.1016/j.neunet.2018.08.025}
  {\path{doi:https://doi.org/10.1016/j.neunet.2018.08.025}}.

\bibitem{JMLR:v20:19-150}
L.~Grigoryeva, J.-P. Ortega,
  \href{http://jmlr.org/papers/v20/19-150.html}{Differentiable reservoir
  computing}, Journal of Machine Learning Research 20~(179) (2019) 1--62.
\newline\urlprefix\url{http://jmlr.org/papers/v20/19-150.html}

\bibitem{mcgoff2015}
K.~McGoff, S.~Mukherjee, N.~Pillai,
  \href{https://doi.org/10.1214/15-SS111}{Statistical inference for dynamical
  systems: A review}, Statist. Surv. 9 (2015) 209--252.
\newblock \href {https://doi.org/10.1214/15-SS111}
  {\path{doi:10.1214/15-SS111}}.
\newline\urlprefix\url{https://doi.org/10.1214/15-SS111}

\bibitem{9174045}
A.~Khaleghi, D.~Ryabko, Clustering piecewise stationary processes, in: 2020
  IEEE International Symposium on Information Theory (ISIT), 2020, pp.
  2753--2758.
\newblock \href {https://doi.org/10.1109/ISIT44484.2020.9174045}
  {\path{doi:10.1109/ISIT44484.2020.9174045}}.

\bibitem{HART2020234}
A.~Hart, J.~Hook, J.~Dawes,
  \href{http://www.sciencedirect.com/science/article/pii/S0893608020301830}{Embedding
  and approximation theorems for echo state networks}, Neural Networks 128
  (2020) 234 -- 247.
\newblock \href {https://doi.org/https://doi.org/10.1016/j.neunet.2020.05.013}
  {\path{doi:https://doi.org/10.1016/j.neunet.2020.05.013}}.
\newline\urlprefix\url{http://www.sciencedirect.com/science/article/pii/S0893608020301830}

\bibitem{Hart_regularised_2020}
A.~G. Hart, J.~L. Hook, J.~H. Dawes,
  \href{https://www.sciencedirect.com/science/article/pii/S0167278921000403}{Echo
  state networks trained by tikhonov least squares are $l^2(\mu)$ approximators
  of ergodic dynamical systems}, Physica D: Nonlinear Phenomena (2021)
  132882\href {https://doi.org/https://doi.org/10.1016/j.physd.2021.132882}
  {\path{doi:https://doi.org/10.1016/j.physd.2021.132882}}.
\newline\urlprefix\url{https://www.sciencedirect.com/science/article/pii/S0167278921000403}

\bibitem{Adaptive_Algorithms_and_Stochastic_Approximations}
A.~Benveniste, M.~Métivier, P.~Prioure, Adaptive Algorithms and Stochastic
  Approximations, Springer-Verlag, 1990.

\bibitem{Borkar2009}
V.~S. Borkar, Stochastic Approximation: A Dynamical Systems Viewpoint,
  Hindustan Book Agency, 2009.

\bibitem{10.1007/978-3-540-72927-3_23}
F.~S. Melo, M.~I. Ribeiro, Q-learning with linear function approximation, in:
  N.~H. Bshouty, C.~Gentile (Eds.), Learning Theory, Springer Berlin
  Heidelberg, Berlin, Heidelberg, 2007, pp. 308--322.

\bibitem{Chen2019}
Z.~Chen, S.~Zhang, T.~T. Doan, S.~T. Maguluri, J.-P. Clarke, Performance of
  q-learning with linear function approximation: Stability and finite-time
  analysis, arXiv:1905.11425 (2019).

\bibitem{LIN20097313}
X.~Lin, Z.~Yang, Y.~Song,
  \href{http://www.sciencedirect.com/science/article/pii/S0957417408006519}{Short-term
  stock price prediction based on echo state networks}, Expert Systems with
  Applications 36~(3, Part 2) (2009) 7313 -- 7317.
\newblock \href {https://doi.org/https://doi.org/10.1016/j.eswa.2008.09.049}
  {\path{doi:https://doi.org/10.1016/j.eswa.2008.09.049}}.
\newline\urlprefix\url{http://www.sciencedirect.com/science/article/pii/S0957417408006519}

\bibitem{doi:10.1260/1748-3018.7.1.87}
H.~Zhang, J.~Liang, Z.~Chai,
  \href{https://doi.org/10.1260/1748-3018.7.1.87}{Stock prediction based on
  phase space reconstruction and echo state networks}, Journal of Algorithms \&
  Computational Technology 7~(1) (2013) 87--100.
\newblock \href {http://arxiv.org/abs/https://doi.org/10.1260/1748-3018.7.1.87}
  {\path{arXiv:https://doi.org/10.1260/1748-3018.7.1.87}}, \href
  {https://doi.org/10.1260/1748-3018.7.1.87}
  {\path{doi:10.1260/1748-3018.7.1.87}}.
\newline\urlprefix\url{https://doi.org/10.1260/1748-3018.7.1.87}

\bibitem{Dan2014}
J.~Dan, W.~Guo, W.~Shi, B.~Fang, T.~Zhang,
  \href{https://doi.org/10.1155/2014/137148}{Deterministic echo state networks
  based stock price forecasting}, Abstract and Applied Analysis 2014 (2014)
  137148.
\newblock \href {https://doi.org/10.1155/2014/137148}
  {\path{doi:10.1155/2014/137148}}.
\newline\urlprefix\url{https://doi.org/10.1155/2014/137148}

\bibitem{6319485}
J.~{Bozsik}, Z.~{Ilonczai}, Echo state network-based credit rating system, in:
  2012 4th IEEE International Symposium on Logistics and Industrial
  Informatics, 2012, pp. 185--190.

\bibitem{6924052}
L.~{Maciel}, F.~{Gomide}, D.~{Santos}, R.~{Ballini}, Exchange rate forecasting
  using echo state networks for trading strategies, in: 2014 IEEE Conference on
  Computational Intelligence for Financial Engineering Economics (CIFEr), 2014,
  pp. 40--47.

\bibitem{doi:10.1080/14697680701381228}
M.~Avellaneda, S.~Stoikov,
  \href{https://doi.org/10.1080/14697680701381228}{High-frequency trading in a
  limit order book}, Quantitative Finance 8~(3) (2008) 217--224.
\newblock \href
  {http://arxiv.org/abs/https://doi.org/10.1080/14697680701381228}
  {\path{arXiv:https://doi.org/10.1080/14697680701381228}}, \href
  {https://doi.org/10.1080/14697680701381228}
  {\path{doi:10.1080/14697680701381228}}.
\newline\urlprefix\url{https://doi.org/10.1080/14697680701381228}

\bibitem{doi:10.1080/1350486X.2017.1342552}
O.~Gu\'eant, \href{https://doi.org/10.1080/1350486X.2017.1342552}{Optimal
  market making}, Applied Mathematical Finance 24~(2) (2017) 112--154.
\newblock \href
  {http://arxiv.org/abs/https://doi.org/10.1080/1350486X.2017.1342552}
  {\path{arXiv:https://doi.org/10.1080/1350486X.2017.1342552}}, \href
  {https://doi.org/10.1080/1350486X.2017.1342552}
  {\path{doi:10.1080/1350486X.2017.1342552}}.
\newline\urlprefix\url{https://doi.org/10.1080/1350486X.2017.1342552}

\bibitem{fazel_global_2018}
M.~Fazel, R.~Ge, S.~Kakade, M.~Mesbahi, Global {Convergence} of {Policy}
  {Gradient} {Methods} for the {Linear} {Quadratic} {Regulator}, in:
  International {Conference} on {Machine} {Learning}, PMLR, 2018, pp.
  1467--1476, iSSN: 2640-3498.

\bibitem{learning_the_LQR}
Z.~Mhammedi, D.~J. Foster, M.~Simchowitz, D.~Misra, W.~Sun, A.~Krishnamurthy,
  A.~Rakhlin, J.~Langford, Learning the linear quadratic regulator from
  nonlinear observations, arXiv:2010.03799 (2020).

\bibitem{dean_sample_2020}
S.~Dean, H.~Mania, N.~Matni, B.~Recht, S.~Tu,
  \href{https://doi.org/10.1007/s10208-019-09426-y}{On the {Sample}
  {Complexity} of the {Linear} {Quadratic} {Regulator}}, Foundations of
  Computational Mathematics 20~(4) (2020) 633--679.
\newblock \href {https://doi.org/10.1007/s10208-019-09426-y}
  {\path{doi:10.1007/s10208-019-09426-y}}.
\newline\urlprefix\url{https://doi.org/10.1007/s10208-019-09426-y}

\bibitem{tu_gap_2019}
S.~Tu, B.~Recht, \href{http://proceedings.mlr.press/v99/tu19a.html}{The {Gap}
  {Between} {Model}-{Based} and {Model}-{Free} {Methods} on the {Linear}
  {Quadratic} {Regulator}: {An} {Asymptotic} {Viewpoint}}, in: Conference on
  {Learning} {Theory}, PMLR, 2019, pp. 3036--3083, iSSN: 2640-3498.
\newline\urlprefix\url{http://proceedings.mlr.press/v99/tu19a.html}

\bibitem{bradtke_adaptive_1994}
S.~Bradtke, B.~Ydstie, A.~Barto, Adaptive linear quadratic control using policy
  iteration, in: Proceedings of 1994 {American} {Control} {Conference} - {ACC}
  '94, Vol.~3, 1994, pp. 3475--3479 vol.3.
\newblock \href {https://doi.org/10.1109/ACC.1994.735224}
  {\path{doi:10.1109/ACC.1994.735224}}.

\bibitem{the_bible}
R.~S. Sutton, A.~G. Barto, Reinforcement Learning: An Introduction, MIT-press,
  2015.

\end{thebibliography}

\end{document}